\documentclass{article}

\usepackage{geometry}
\usepackage{float}
\usepackage{amsmath}
\usepackage{amssymb}
\usepackage{amsfonts}
\usepackage{amsthm}

\usepackage{enumitem}

\usepackage{longtable}

\usepackage{tikz}
    \usetikzlibrary{matrix,arrows,decorations.markings,backgrounds}

\newcommand{\NN}{\mathbb{N}}

\newcommand{\ZZ}{\mathbb{Z}}

\newcommand{\RR}{\mathbb{R}}

\newcommand{\CC}{\mathbb{C}}

\newcommand{\HH}{\mathbb{H}}
\newcommand{\kk}{\mathfrak{k}}

\renewcommand{\phi}{\varphi}
\renewcommand{\epsilon}{\varepsilon}

\newcommand{\so}{\mathfrak{so}}
\newcommand{\lieg}{\mathfrak g}

\DeclareMathOperator{\SO}{SO} 
\DeclareMathOperator{\hht}{ht} 
\DeclareMathOperator{\Fix}{Fix}

\DeclareMathOperator{\PSL}{PSL}

\DeclareMathOperator{\ad}{ad}

\newcommand{\liesl}{\mathfrak{sl}}

\newcommand{\imp}{\mathfrak{m}}

\DeclareMathOperator{\im}{im}

\DeclareMathOperator{\Aut}{Aut}     
     
\DeclareMathOperator{\id}{id}

\DeclareMathOperator{\End}{End}

\DeclareMathOperator{\Sym}{Sym}

\DeclareMathOperator{\SL}{SL}

\DeclareMathOperator{\mfg}{\mathfrak g}

\newtheorem{proposition}{Proposition}[section]
\newtheorem{theorem}[proposition]{Theorem}
\newtheorem{lemma}[proposition]{Lemma}
\newtheorem{corollary}[proposition]{Corollary}

\newtheorem{consequence}[proposition]{Consequence}

\newtheorem{innercustomthm}{Theorem}
\newenvironment{customthm}[1]
  {\renewcommand\theinnercustomthm{#1}\innercustomthm}
  {\endinnercustomthm}

\theoremstyle{remark}
\newtheorem{example}[proposition]{Example}
\newtheorem{construction}[proposition]{Construction}
\newtheorem{remark}[proposition]{Remark}

\theoremstyle{definition}
\newtheorem{definition}[proposition]{Definition}

\renewcommand{\labelenumi}{{\rm(\alph{enumi})}}
\renewcommand{\labelenumii}{{\rm(\roman{enumii})}}

\newcommand{\M}[2]{M({#1},{#2})}

\def\DynkinArrowLength{2mm}

\tikzset{
  dnode/.style={
    circle,
    inner sep=0pt,
    minimum size=1.5mm,	
    fill=black,			
    draw},
  middlearrow/.style={
    decoration={markings,
      mark=at position 0.6 with
      {\draw (0:0mm) -- +(+135:\DynkinArrowLength); \draw (0:0mm) -- +(-135:\DynkinArrowLength);},
    },
    postaction={decorate}
  },
  leftrightarrow/.style={
    decoration={markings,
      mark=at position 0.999 with
      {
      \draw (0:0mm) -- +(+135:\DynkinArrowLength); \draw (0:0mm) -- +(-135:\DynkinArrowLength);
      },
      mark=at position 0.001 with
      {
      \draw (0:0mm) -- +(+45:\DynkinArrowLength); \draw (0:0mm) -- +(-45:\DynkinArrowLength);
      },
    },
    postaction={decorate}
  },
  sedge/.style={
  },
  dedge/.style={
    middlearrow,
    double distance=0.5mm,
  },
  tedge/.style={
    middlearrow,
    double distance=1.0mm+\pgflinewidth,
    postaction={draw}, 
  },
  infedge/.style={
    leftrightarrow,
    double distance=0.5mm,
  },
}

\newcommand{\En}[1]{
\begin{tikzpicture}
    \node[dnode,label=below:$1$] (1) at (0,0) {};
    \node[dnode,label=above:$2$] (2) at (2,1) {};
    \ifx3#1
    \else
	    \path (2) edge[sedge] (4);
	\fi

    \foreach \x in {3,...,#1}
    {
        \node[dnode,label=below:{$\x$}] (\x) at (\x - 2,0) {};
        \draw[sedge] (\x - 3,0) -- +(1,0);
    }
\end{tikzpicture}
}

\begin{document}

\title{Generalized spin representations \\ \vspace{.35cm} {\large With an appendix by Max Horn and Ralf K\"ohl: \\ Cartan--Bott periodicity for the real $E_n$ series}}
\author{Guntram Hainke and Ralf K\"ohl and Paul Levy}

\maketitle

\begin{abstract}
We introduce the notion of a generalized spin representation of the maximal compact subalgebra $\mathfrak k$ of a 
symmetrizable Kac--Moody algebra $\mathfrak g$ in order to show that, if defined over a formally real field, every such $\mathfrak k$ has a 
non-trivial reductive finite-dimensional quotient. The appendix illustrates how to compute the isomorphism types of these quotients for the real $E_n$ series. In passing this provides an elementary way of determining the isomorphism types of the maximal compact subalgebras of the semisimple split real Lie algebras of types $E_6$, $E_7$, $E_8$.
\end{abstract}


\section{Introduction}
During the last decade the family of Kac--Moody algebras of type $E_{n}(\RR)$ 
has received considerable attention because of its importance in M-theory \cite{de2006kac}, \cite{gebert}, \cite{Kleinschmidt/Nicolai/Palmkvist}, \cite{palmkvist2009exceptional}, \cite{west2001e11}. 
By \cite{DamourKleinschmidtNicolai}, \cite{deBuylHenneauxPaulot} the (so-called) maximal compact subalgebra $\mathfrak k=\Fix \omega$ of the real split Kac--Moody algebra $\mathfrak g=\mathfrak{g}(E_{10})(\RR)$ with respect to the Cartan--Chevalley involution $\omega$ admits a 32-dimensional complex representation which extends the spin representation of its regular subalgebra $\mathfrak{so}_{10}(\RR)$. 
This implies that the (infinite-dimensional) Lie algebra $\mathfrak k$ has a non-trivial finite-dimensional quotient, in fact a semisimple finite-dimensional quotient (see Theorem~\ref{Mainexistencetheorem}). Since $\mathfrak k$ is anisotropic with respect to the invariant bilinear form of the Kac--Moody algebra $\mathfrak{g}$, it actually contains an ideal isomorphic to this finite-dimensional quotient.

\medskip
In this article we show that the existence of non-trivial finite-dimensional representations is not peculiar 
to the maximal compact subalgebra of $\mathfrak{g}(E_{10})(\mathbb{R})$ but is shared by all maximal compact subalgebras of symmetrizable Kac--Moody algebras over arbitrary fields of characteristic $0$.
To this end we introduce the notion of a generalized spin representation (Definitions \ref{genspinrep} and \ref{genspinrep2}), which we inductively show to exist for arbitrary symmetrizable Kac-Moody algebras and which, in the case of formally real fields, affords a compact, whence reductive, and often even a semisimple image (Theorem~\ref{Mainexistencetheorem}).

\medskip
Our results presented in this article are generalizations of the results concerning the $\frac{1}{2}$-spin representations described in \cite{DamourKleinschmidtNicolai}, \cite{deBuylHenneauxPaulot}. The key observation is Remark~\ref{characterization} that in the simply-laced case a $\frac{1}{2}$-spin representation can be described by linear operators $A_i$ for each vertex $i$ of the diagram that satisfy
\begin{enumerate}
\item[(i)] $A_i^2 = -\frac{1}{4} \cdot \id$,
\item[(ii)] $A_iA_j = A_jA_i$, if the vertices $i$, $j$ do not form an edge of the diagram,
\item[(iii)] $A_iA_j=-A_jA_i$, if the vertices $i$, $j$ form an edge of the diagram.
\end{enumerate}
On the other hand, the $\frac{3}{2}$-spin representations of \cite{DamourKleinschmidtNicolai}, \cite{deBuylHenneauxPaulot} and the $\frac{5}{2}$- and $\frac{7}{2}$-spin representations of \cite{Kleinschmidt/Nicolai} are still elusive, as the algebraic identities that need to be satisfied by the corresponding linear operators are more involved.

\medskip
Note that our terminology of {\em maximal compact subalgebra} is misleading. For one, in the infinite-dimensional situation there is no compact group associated to a maximal compact subalgebra. Rather, over the real numbers, the maximal compact subalgebra is related to the group $K$ studied in \cite{KacPeterson}, \cite{Medts/Gramlich/Horn}. This group naturally carries a non-locally compact non-metrizable $k_\omega$-topology (cf.\ \cite{Hartnick/Koehl/Mars}). Moreover, our construction only involves the Cartan--Chevalley involution and no field involution. Therefore, over the complex numbers, what we call a maximal compact subalgebra is not even anisotropic.

However, this terminology does not lead to serious ambiguities as our main focus lies on split Lie algebras over formally real fields. Our main structure-theoretic results in Section~\ref{GSR} below will consequently be obtained over formally real fields; the main future application of our result is over the real numbers.

\bigskip \noindent
\textbf{Acknowledgements.} We thank Pierre-Emmanuel Caprace for pointing out to us the $32$-dimensional representation of 
the maximal compact subalgebra of $E_{10}(\mathbb{R})$, thus triggering our research. We also thank Kay Magaard for bringing our attention to \cite{Maas} and Thibault Damour, David Ghatei, Axel Kleinschmidt, Karl-Hermann Neeb, Sebastian Wei\ss\, and especially Max Horn and two anonymous referees for valuable comments on preliminary versions of this work. This research has been partially funded by the EPRSC grants 
EP/H02283X and EP/K022997/1. The second author gratefully acknowledges the hospitality of the IHES at Bures-sur-Yvette and of the Albert Einstein Institute at Golm.

\section{Preliminaries}

In this section we collect several basic facts about Kac--Moody algebras. We refer the reader to 
\cite[Chapter 1]{kac1994infinite} and \cite[Chapter 1]{MR1923198} for proofs and further details.

\subsection{Kac--Moody algebras}\label{KMsubsec}
Let $k$ be a field of characteristic 0, let $A=(a_{ij}) \in \mathbb{Z}^{n \times n}$ be a \textbf{generalized Cartan matrix}
and let $\mathfrak{g}=\mathfrak{g}_A$ denote the corresponding \textbf{Kac--Moody algebra} over $k$. This means that 
$$a_{ii}=2, \quad a_{ij}\leq 0 \quad \mbox{and} \quad a_{ij}=0 \Leftrightarrow a_{ji}=0,$$ while 
$\mathfrak{g}$ is the quotient 
of 
the free 
Lie algebra over $k$ generated by $e_i$, $f_i$, $h_i$, $i=1, \ldots, n$, subject to the relations
$$ [h_i,h_j]=0,\; [h_i,e_j]=a_{ij}e_j,\; [h_i,f_j]=-a_{ij}f_j \text{ for all } 1 \leq i, j \leq n, $$
$$ [e_i,f_j]=0,\;[e_i,f_i]=h_i,\; (\ad e_i)^{-a_{ij}+1}(e_j)=0, (\ad f_i)^{-a_{ij}+1}(f_j)=0 \text { for } i \neq j.$$
A generalized Cartan matrix is called {\bf simply laced} if the off-diagonal entries of $A$ are either 0 or $-1$; it is called {\bf symmetrizable} if there exists a diagonal matrix $\Lambda$ such that $\Lambda A$ is symmetric. \\
By abuse of terminology, we will say that ${\mathfrak g}$ is simply laced, resp.\ symmetrizable if its generalized Cartan matrix is simply laced, resp.\ symmetrizable.

Let $\mathfrak h:=\langle h_1, \ldots, h_n \rangle$, $\mathfrak n_+:=\langle e_1, \ldots, e_n \rangle$ and 
$\mathfrak n_-:=\langle f_1,\ldots, f_n \rangle$ denote the standard subalgebras of $\mathfrak g$. Then there is a 
decomposition as vector spaces $$\mathfrak g=\mathfrak n_- \oplus \mathfrak h \oplus \mathfrak n_+$$
 (see 
\cite[\S1.3, p.~7]{kac1994infinite}). The defining relations of $\mathfrak{g}$ imply that $\mathfrak{h}$ is $n$-dimensional abelian and normalizes $\mathfrak{n}_+$ and $\mathfrak{n}_-$. In fact, it acts by linear transformations on these vector spaces. Therefore, for each element $\alpha \in \mathfrak{h}^*$ of the dual space it is meaningful to define the eigenspaces $$\mathfrak{g}_\alpha := \{ x \in \mathfrak{g} \mid \forall h \in \mathfrak{h} : [h,x] = \alpha(h)x \}.$$ 
The relations $[h_i,e_j]=a_{ij}e_j$, $1 \leq i, j \leq n$, imply that each $e_j$ is contained in such an eigenspace, which we denote by $\mathfrak{g}_{\alpha_j}$; the corresponding element of $\mathfrak{h}^*$ is denoted by $\alpha_j$. (Cf.\
\cite[\S1.1]{kac1994infinite}.) Note that $\mathfrak{g}_{-\alpha_j}$ contains $f_j$. 

The \textbf{diagram} of a simply laced Kac--Moody algebra $\mathfrak g_A$ is the graph $D=(V,E)$ on vertices $\alpha_1, \ldots, \alpha_n$ with $\alpha_i$ and $\alpha_j$ connected by an edge if and only if $a_{ij}=-1$.

Let $Q:=\oplus_{i=1}^n \ZZ \alpha_i$ denote a free 
$\ZZ$-module of rank $n$ and $Q_+:=\oplus_{i=1}^n \ZZ_+ \alpha_i$, where the latter denotes the set of non-negative integral linear combinations. By 
\cite[Thm.~1.2(d), Exercise~1.2]{kac1994infinite} $$\mathfrak{g} = \bigoplus_{\alpha \in Q} g_\alpha = \mathfrak{h} \oplus \bigoplus_{\alpha \in Q \backslash \{ 0 \}} g_\alpha = \bigoplus_{\alpha \in Q_+ \backslash \{ 0 \}} g_{-\alpha} \oplus \mathfrak{h} \oplus \bigoplus_{\alpha \in Q_+ \backslash \{ 0 \}} g_\alpha.$$
Therefore, $\mathfrak g$ has a $Q$-grading by declaring 
$$\deg 
h_i:=0, \quad \deg e_i:=\alpha_i, \quad \deg f_i:=-\alpha_i$$ for $i=1, \ldots, n$, i.e., $$\mathfrak g=\bigoplus_{\alpha \in 
Q} 
\mathfrak g_\alpha \quad \mbox{and} \quad [\mathfrak g_\alpha, \mathfrak g_\beta]\subseteq \mathfrak 
g_{\alpha+\beta}.$$
Let $\Delta:=\{\alpha \in Q \backslash\{0\} \mid \mathfrak g_\alpha \neq 0\}$. Then $\Delta=\Delta_+ \cup \Delta_-$, where $\Delta_+:=\Delta \cap (Q_+ \backslash \{ 0 \})$ and $\Delta_-:=-\Delta_+$. An element $\alpha \in \Delta$ is called a \textbf{root} and $\mathfrak g_\alpha$ a \textbf{root space}. A root $\alpha \in \Delta$ is called \textbf{positive} if it belongs to $\Delta_+$, otherwise \textbf{negative}. A root of the form $\alpha=\pm \alpha_i$ is called \textbf{simple}.\\

Since the adjoint representation $\mathrm{ad} : \mathfrak{g} \to \mathrm{End}(\mathfrak{g})$ is integrable (see \cite[\S 3.5]{kac1994infinite}), the \textbf{extended Weyl group} $W^* \leq \Aut \mathfrak g$ can be defined as $W^*:=\langle s_i^* \mid i=1, \ldots, n \rangle$, where 
$$s_i^* := s_i^{\mathrm{ad}} :=\exp \ad f_i \cdot \exp \ad (-e_i) \cdot \exp \ad f_i$$ (cf.\ \cite[\S 3.8]{kac1994infinite}; note that $W^* \leq \Aut \mathfrak g$ by \cite[Lem.~3.8(b)]{kac1994infinite}
).
For $\alpha \in \Delta$ and $w \in W^*$ there exists a unique $w\cdot \alpha \in \Delta$ such that $w(\mathfrak{g}_\alpha)=\mathfrak g_{w \cdot \alpha}$, by \cite[Lem.~3.8(a)]{kac1994infinite}. A root $\alpha$ is called \textbf{real} if there is a $w \in W$ such that $w\cdot\alpha$ is simple, otherwise it is called \textbf{imaginary}. Let $\Delta^{\text{re}}$ denote the set of real roots and $\Delta^{\text{im}}$ the set of imaginary roots.
 
For $\alpha=\sum_{i=1}^n{a_i \alpha_i} \in  \Delta$, the {\bf height} of $\alpha$ is defined as $\hht \alpha:=\sum\limits_{i=1}^n a_i$. For $n \in \NN$ let 
$$(\mathfrak n_+)_n:=\bigoplus\limits_{\substack{ \alpha \in \Delta^+\\  \hht \alpha=n}} \mfg_\alpha.$$ 
This is a $\ZZ$-grading of $\mathfrak n_+$ and extends to a $\ZZ$-grading of $\mathfrak{g}$, the {\bf principal grading} (cf.\ 
\cite[\S1.5]{kac1994infinite}). 

\subsection{The maximal compact subalgebra}
Let $\mathfrak g$ be a Kac--Moody algebra over a field $k$ of characteristic 0. 
Let $\omega \in \Aut(\mathfrak g)$ denote the {\bf Cartan--Chevalley involution} characterized by 
$\omega(e_i)=-f_i$,  $\omega(f_i)=-e_i$ and $\omega(h_i)=-h_i$. (Cf.\ 
\cite[Equ.~(1.3.4)]{kac1994infinite}.) Observe that $\omega(\mathfrak{g}_\alpha) = \mathfrak{g}_{-\alpha}$.

Let $\mathfrak k:=\mathfrak{k}(\mathfrak{g}):=\left\{X \in \mathfrak{g} \mid \omega(X)=X\right\}$ denote the fixed point subalgebra, which --- in analogy to the situation of finite-dimensional semisimple split real Lie algebras --- is called the \textbf{maximal compact subalgebra} of $\mathfrak{g}$.
For example, if $\mathfrak g=\mathfrak {sl}_n(\RR)$, then $\omega(A)=-A^T$ and $\mathfrak 
k=\mathfrak{so}_n(\RR)$. In this case, $\mathfrak {so}_n(\RR)$ is the Lie algebra of the maximal compact subgroup 
$\SO_n(\RR)$ of $\SL_n(\RR)$. See also \cite[Section IV.4]{knapp}. 

Over non-real closed fields, especially over the complex numbers, our terminology is a bit unfortunate and misleading. However, our main results in Section~\ref{GSR} below and future applications are over real closed fields.

A theorem of Berman \cite{Berman} allows one to give a presentation of these. We point out that Berman's result in fact deals with a much more general class of so-called involutory algebras by also allowing other involutions of $\mathfrak{g}$ of the second kind (in the sense of \cite[4.6]{Kac/Wang:1992}). Note that Berman instead of our involution $\omega$ uses the involution $\eta$ given by $\eta(e_i)=f_i$, $\eta(f_i)=e_i$, $\eta(h_i)=-h_i$ as the foundation of his investigations so that in order to apply his result one still has to relate the two involutions to one another. 

\begin{theorem}[{cf.\ \cite[Thm.~1.31]{Berman}}] \label{Berman} Let $k$ be a field of characteristic 0. Let $A 
\in 
\mathbb{Z}^{n \times n}$ be a 
simply laced generalized Cartan matrix, let $\mathfrak{g}_A$ denote the corresponding Kac--Moody algebra and let $\mathfrak k$ denote the maximal compact subalgebra of $\mathfrak g$. \\  
Then $\mathfrak k$ is isomorphic to the quotient of the free Lie algebra over $k$ generated by $X_1, \ldots, X_n$ subject to the relations
$$ \begin{array}{rcll}
[X_i,[X_i,X_j]] & = & -X_j, & \text{if the vertices $v_i,v_j$ are connected by an edge,} \\
\text{$[X_i, X_j]$}      & = & 0, & \text{otherwise,}\\
\end{array}
$$
via the map $X_i \mapsto e_i-f_i$.
\end{theorem}

In Theorem~\ref{Berman2} below we state and prove a general version of this result that applies to the maximal compact subalgebra of an arbitrary symmetrizable Kac--Moody algebra over a field of characteristic $0$. Our motivation for splitting off the simply-laced case is that it is considerably easier to understand than the general case. Furthermore, the study of generalized spin representations in the simply-laced case is key to these representations in general. 

\begin{proof}[Proof of Theorem~\ref{Berman}.]
Let $\eta \in \Aut \mathfrak g$ denote the involution characterized by $$\eta(e_i)=f_i,\,  \eta(f_i)=e_i \text{ and } \eta(h_i)=-h_i$$ and let $\mathfrak l:=\Fix \eta$ denote the subalgebra of fixed points of $\eta$. 
By \cite[Thm.~1.31]{Berman}, the Lie algebra $\mathfrak l$ is isomorphic to the quotient of the free Lie algebra over $k$ generated by $Y_1, \ldots, Y_n$ subject to the relations 
$$ \begin{array}{rcll}
[Y_i,[Y_i,Y_j]] & = & Y_j, & \text{if the vertices $v_i,v_j$ are connected by an edge,} \\
\text{$[Y_i, Y_j]$}      & = & 0, & \text{otherwise,}\\
\end{array}
$$
via the map $Y_i \mapsto e_i+f_i$.

Let $I:=\sqrt{-1}$ denote a square root of $-1$ and let $L:=k(I)$, $\mathfrak g_L:=\mathfrak g \otimes_k L$.  
There is a Lie algebra automorphism $\varphi \in \Aut(\mathfrak g_L)$ determined by
$$e_i \mapsto I\cdot e_i, \, f_i \mapsto -I \cdot f_i \text{ and } h_i \mapsto h_i.$$ This automorphism $\varphi$ conjugates $\eta$ to $\omega$, 
i.e.\ $\omega=\phi^{-1} \circ \eta \circ \phi$, and hence the subalgebras $\Fix \eta$ and $\Fix \omega$ are 
isomorphic over $L$. As $X_i$ is mapped to $I\cdot Y_i$ under this isomorphism, the claim follows.
\end{proof}

\begin{remark}\label{weylgroupremark}
Suppose $k={\mathbb C}$.
We can exponentiate the subalgebra of ${\mathfrak g}$ spanned by $e_i,f_i,h_i$ to a subgroup $G_i$ of $\Aut{\mathfrak g}$ which is isomorphic to $\SL_2({\mathbb C})$ or $\PSL_2({\mathbb C})$.
Then $X_i$ identifies with $\begin{pmatrix} 0 & 1 \\ -1 & 0 \end{pmatrix}$ in $\mathfrak{sl}_2$ and therefore $\exp(\xi X_i)$ is equal to the image of $\begin{pmatrix} \cos\xi & \sin\xi \\ -\sin\xi & \cos\xi\end{pmatrix}$ in $G_i$.
In particular, $\exp(-\frac{\pi}{2}X_i)$ is sent to $s_i^*$.
It follows that $s_i^*$ and $\omega$ are commuting automorphisms of ${\mathfrak g}$.

For the case of an arbitrary ground field, $\omega$ induces a Cartan--Chevalley involution on the standard type $A_1$ subgroup $G_i$ of $\Aut{\mathfrak g}$ whose Lie algebra is spanned by $e_i$, $f_i$, $h_i$.
The fixed point subgroup of $G_i$ for the Cartan--Chevalley involution is either $\SO_2(k)$ or $\SO_2(k)/\{\pm I_2\}$, depending on whether $G_i$ is isomorphic to $\SL_2$ or $\PSL_2$.
Since this subgroup clearly contains $s_i^*$, it follows that $s_i^*$ commutes with $\omega$.
\end{remark}

\subsection{Rank 2 Kac--Moody algebras} \label{rank2KM}
Let ${\mathfrak g}$ be the Kac--Moody algebra with Cartan matrix $\begin{pmatrix} 2 & -r \\ -s & 2 \end{pmatrix}$, where $r,s\in{\mathbb N}$.
We map ${\mathfrak g}$ into a simply laced Kac--Moody algebra as follows: Let $D$ be a complete bipartite graph on $r$ and $s$ vertices, labelled $\alpha_1^{(i)}$ and $\alpha_2^{(j)}$ with $1\leq i\leq r$, $1\leq j\leq s$.
Let $\tilde{\mathfrak g}$ be a Kac--Moody Lie algebra with simply laced diagram $D$ and label the generators correspondingly: $e_1^{(i)}$, $f_1^{(i)}$, $h_1^{(i)}$ and $e_2^{(j)}$, $f_2^{(j)}$, $h_2^{(j)}$.
We remark that there is an action of $\mathrm{Sym}(r)$ (resp. $\mathrm{Sym}(s)$) on $\tilde{\mathfrak g}$ by permuting the roots $\alpha_1^{(i)}$ (resp. $\alpha_2^{(j)}$).
Let
\begin{eqnarray*}
E_1=\sum_{i=1}^r e_1^{(i)},\;\; F_1=\sum_{i=1}^r f_1^{(i)},\;\; H_1=[E_1,F_1], \\ E_2=\sum_{j=1}^s e_2^{(j)},\;\; F_2=\sum_{j=1}^s f_2^{(j)},\;\; H_2=[E_2,F_2].
\end{eqnarray*}
Then it is straightforward to check that $[E_1,F_2]=0=[E_2,F_1]=[H_1,H_2]$, $(\ad E_1)^{r+1}(E_2)=0=(\ad E_2)^{s+1}(E_1)$, and $(\ad F_1)^{r+1}(F_2)=(\ad F_2)^{s+1}(F_1)=0$.
Thus there is a well-defined Lie algebra homomorphism $\tilde\varphi$ from ${\mathfrak g}$ to $\tilde{\mathfrak g}$, sending each of $e_1$, $e_2$, $f_1$, $f_2$, $h_1$, $h_2$ to its corresponding upper-case letter.
Since ${\mathfrak g}$ has no non-zero ideals intersecting trivially with ${\mathfrak h}$, it follows that $\tilde\varphi$ is injective.
It is clear from the definitions that $\tilde\varphi$ induces an injective homomorphism from the extended Weyl group of ${\mathfrak g}$ to that of $\tilde{\mathfrak g}$ by sending $s_1^*$ to $(s_1^{(1)})^*\ldots (s_1^{(r)})^*$, and similarly for $s_2^*$.

\begin{remark}
This construction is related to the notion of {\it pinning}\footnote{French ``\'epinglage'', see \cite[Expos\'e XXIII]{sga}. Although this is translated as ``framing'' in \cite{Bourbaki}, it is clear from the footnote to \cite[Expos\'e XXIII, Def.~1.1]{sga} (where a maximal torus is the body, and opposite Borel subgroups are the wings, of a butterfly) that ``pinning'' is more appropriate.
It seems to have become the standard terminology in English.} for split semisimple Lie algebras.
Given a split semisimple Lie algebra $\tilde{\mathfrak g}$ over a field $k$ of characteristic zero, let $\tilde{\mathfrak h}$ be a splitting Cartan subalgebra.
A pinning of $(\tilde{\mathfrak g},\tilde{\mathfrak h})$ consists of a basis $\Pi$ of the roots of $\tilde{\mathfrak g}$ relative to $\tilde{\mathfrak h}$, together with a choice $\{ x_\alpha : \alpha\in \Pi\}$ of non-zero elements in each simple positive root space.
If $\tilde{\mathfrak g}$ has a presentation as in \S \ref{KMsubsec} then we can take $\Pi=\{ \alpha_1,\ldots ,\alpha_n\}$ and $x_{\alpha_i}=e_i$ for $1\leq i\leq n$.
If a pinning of $(\tilde{\mathfrak g},\tilde{\mathfrak h})$ is fixed, then a {\it pinned automorphism} is an automorphism which stabilizes $\tilde{\mathfrak h}$ and the Borel subalgebra of $\tilde{\mathfrak g}$ corresponding to $\Pi$, and which permutes the elements $x_\alpha$, $\alpha\in\Pi$.
Clearly, the group of pinned automorphisms is isomorphic to the group ${\rm Aut}(\Pi)$ of automorphisms of the Dynkin diagram of $\tilde{\mathfrak g}$.
As follows from \cite[VIII.3 Cor.~1 and VIII.4]{Bourbaki}, the group ${\rm Aut}(\tilde{\mathfrak g})$ is the semidirect product of ${\rm Aut}(\Pi)$ and $\tilde{G}(k)$, where $\tilde{G}$ is the adjoint type semisimple group with Lie algebra $\tilde{\mathfrak g}$.
The corresponding result is also true in the Kac--Moody case \cite[\S 6, Theorem~2(c)]{Peterson-Kac}.
When $\tilde{\mathfrak g}$ has generalized Cartan matrix $\begin{pmatrix} 2 & -r \\ -s & 2\end{pmatrix}$, one obtains that the automorphism group is $({\rm Sym}(r)\times{\Sym} (s))\ltimes \tilde{G}$ if $r\neq s$ and is $({\rm Sym}(r) \wr{\rm Sym}(2))\ltimes\tilde{G}$ if $r=s$, where $\tilde{G}$ is an adjoint Kac--Moody group corresponding to $\tilde{\mathfrak g}$.
(We exclude here the affine cases $r=s=2$ and $\{ r,s\}=\{ 1,4\}$, where the picture is slightly more complicated.)

If $\tilde{\mathfrak g}$ has finite type, then there are no non-trivial pinned automorphisms unless $\tilde{\mathfrak g}$ is simply laced.
Furthermore, a simple Lie algebra of type $B_n$ (resp.\ $C_n$, $F_4$, $G_2$) can be realised as the fixed point subalgebra for a pinned automorphism of a Lie algebra of type $D_{n+1}$ (resp.\ $A_{2n-1}$, $E_6$, $D_4$).
In our case we can only say that ${\mathfrak g}$ is a {\it subalgebra} of the fixed-point subalgebra of $\tilde{\mathfrak g}$.
\end{remark}

Let $\tilde\omega$ (resp.\ $\omega$) denote the Cartan--Chevalley involution on $\tilde{\mathfrak g}$ (resp.\ ${\mathfrak g}$).
Clearly $\tilde\varphi\circ\omega=\tilde\omega\circ\tilde\varphi$, so $\tilde\varphi$ induces a homomorphism from ${\mathfrak k}={\mathfrak k}({\mathfrak g})$ to $\tilde{\mathfrak k}={\mathfrak k}(\tilde{\mathfrak g})$.
Following the proof of Theorem~\ref{Berman}, let $Y_1=e_1+f_1$, $Y_2=e_2+f_2$, $Y_1^{(i)}=e_1^{(i)}+f_1^{(i)}$ and $Y_2^{(j)}=e_2^{(j)}+f_2^{(j)}$ for $1\leq i\leq r$, $1\leq j\leq s$.
Then $\tilde\varphi(Y_1)=\sum_1^r \tilde{Y}_1^{(i)}$ and similarly for $Y_2$.

Since $\alpha_1^{(i)}$ and $\alpha_2^{(j)}$ are connected by a simple edge, we have $((\ad Y_1^{(i)})^2-1)(Y_2^{(j)})=0$.
Now the space spanned by $Y_1^{(i)}$ for $1\leq i\leq r$ is conjugate to the subspace of $\tilde{\mathfrak h}$ spanned by $h_1^{(i)}$ for $1\leq i\leq r$.
Thus the fact that $((\ad Y_1^{(i)})^2-1)(Y_2^{(j)})=0$ can be restated by saying that $Y_2^{(j)}$ is a sum of simultaneous eigenvectors for $\ad Y_1^{(i)}$, with each such eigenvalue being $\pm 1$.
It follows that $Y_2^{(j)}$ is contained in the sum of eigenspaces for $\ad \tilde\varphi(Y_1)$ in $\tilde{\mathfrak g}$ with eigenvalues $r, r-2, \ldots ,-r$.
Hence $$\left(\prod_{i=0}^r (\ad\tilde\varphi(Y_1)-(r-2i))\right)(\tilde\varphi(Y_2))=0.$$
Setting $X_i=e_i-f_i$ for $i=1,2$ and conjugating $Y_i$ to $X_i$ as in the proof of Theorem~\ref{Berman}, we deduce that $P_{r}(\ad X_1)(X_2)=0$ and $P_s(\ad X_2)(X_1)=0$, where $$P_m(t)= \left\{ \begin{array}{rl} (t^2+m^2)(t^2+(m-2)^2)\cdots (t^2+1), & \text{if $m$ is odd}, \\ (t^2+m^2)(t^2+(m-2)^2)\cdots (t^2+4)t, & \text{if $m$ is even}. \end{array} \right.$$

\subsection{The general symmetrizable case}\label{symmetrizablesec}

Now suppose ${\mathfrak g}$ is an arbitrary symmetrizable Kac--Moody algebra with $n\times n$ generalized Cartan matrix $A=(a_{ij})_{1\leq i,j\leq n}$.
For $1\leq i\leq n$ let $X_i=e_i-f_i\in{\mathfrak k}$.
On restricting to the rank 2 subalgebra of ${\mathfrak g}$ generated by $e_i,e_j,f_i,f_j$ we obtain the relation $P_{-a_{ij}}(\ad X_i)(X_j)=0$.
As in the simply-laced case, we can use Berman's Theorem \cite[Thm.~1.31]{Berman} to prove that these generate all of the relations in ${\mathfrak k}$.
We reproduce a proof (which also applies in the simply-laced case) for the sake of completeness.

\begin{theorem}\label{Berman2}
The maximal compact subalgebra ${\mathfrak k}$ of ${\mathfrak g}$ has generators $X_1$, \ldots, $X_n$ and relations: $$\left(P_{-a_{ij}}(\ad X_i)\right) (X_j)=0$$ for any $1\leq i\neq j\leq n$.
\end{theorem}

\begin{proof}
By the Gabber--Kac Theorem \cite[Thm.~9.11]{kac1994infinite} the ideal of relations satisfied by $e_1,\ldots ,e_n$ is generated by the terms $(\ad e_i)^{-a_{ij}+1}(e_j)=0$.
Let ${\mathcal L}$ be the Lie algebra on generators $x_1,\ldots ,x_n$ with relations $P_{-a_{ij}}(\ad x_i)(x_j)=0$ for $1\leq i\neq j\leq n$.
Then there is a Lie algebra homomorphism $\pi:{\mathcal L}\rightarrow{\mathfrak k}$, sending $x_i$ to $X_i=e_i-f_i$.

For $\alpha,\beta\in Q_+$ we write $\alpha\leq \beta$ when $\beta-\alpha\in Q_+$.
We note that both ${\mathcal L}$ and ${\mathfrak k}$ are {\it filtered} by $Q_+$, that is, there exist subspaces ${\mathcal L}_{(\alpha)}$ of ${\mathcal L}$ such that:

 - ${\mathcal L}=\cup_{\alpha\in Q_+}{\mathcal L}_{(\alpha)}$;

 - ${\mathcal L}_{(\alpha)}\subset{\mathcal L}_{(\beta)}$ whenever $\alpha\leq\beta$; and

 - $[{\mathcal L}_{(\alpha)},{\mathcal L}_{(\beta)}]\subseteq{\mathcal L}_{(\alpha+\beta)}$;

and similarly for ${\mathfrak k}$.
Specifically, ${\mathfrak k}_{(\alpha)}=(\sum_{-\alpha\leq\beta\leq\alpha}{\mathfrak g}_\beta)\cap{\mathfrak k}$ and ${\mathcal L}_{(\alpha)}$ is the span of all commutators $$[x_{i_1},[x_{i_2},[\ldots[x_{i_{r-1}},x_{i_r}]\ldots ]]$$ where $\alpha_{i_1} +\ldots +\alpha_{i_r}\leq\alpha$.
These filtrations are compatible, i.e.\ $\pi({\mathcal L}_{(\alpha)})\subset{\mathfrak k}_{(\alpha)}$.
For $\alpha\in Q_+$, let ${\mathcal L}_{<\alpha}:=\sum_{\beta<\alpha}{\mathcal L}_{(\beta)}$ and similarly for ${\mathfrak k}$.
The {\it corresponding graded Lie algebra} of ${\mathcal L}$ is the vector space $${\rm gr}\,{\mathcal L}:=\sum_{\alpha\in Q_+}{\mathcal L}_{(\alpha)}/{\mathcal L}_{<\alpha}$$
with the Lie bracket induced by that on ${\mathcal L}$.
For $1\leq i\leq n$ let $\overline{x}_i$ denote the image of $x_i$ in ${\mathcal L}_{(\alpha_i)}/{\mathcal L}_{<\alpha_i}\subset{\rm gr}\, {\mathcal L}$.
By the definition of the polynomials $P_m$, we have $(\ad \overline{x}_i)^{-a_{ij}+1}(\overline{x}_j)=0$ for $1\leq i\neq j\leq n$.
It follows that there is a surjective homomorphism ${\mathfrak n}_+\rightarrow{\rm gr}\, {\mathcal L}$ sending $e_i$ to $\overline{x}_i$.
On the other hand, ${\mathfrak k}_{(\alpha)}/{\mathfrak k}_{<\alpha}$ is spanned by $({\mathfrak g}_\alpha\oplus{\mathfrak g}_{-\alpha})\cap{\mathfrak k}$ so is of dimension $\dim{\mathfrak g}_\alpha$.
(In fact, ${\rm gr}\, {\mathfrak k}\cong{\mathfrak n}_+$, see the remarks after Proposition~\ref{contractionprop} below.)

Now we can prove the theorem as follows.
First of all, we claim that the homomorphism $\pi:{\mathcal L}\rightarrow{\mathfrak k}$ is surjective.
To prove our claim it will suffice to show that $\pi({\mathcal L}_{(\alpha)})={\mathfrak k}_{(\alpha)}$ for all $\alpha\in\Delta_+$.
We note that ${\mathfrak g}_\alpha$ is spanned by elements of the form $y_\alpha=[e_i,y_{\alpha-\alpha_i}]$ where $y_{\alpha-\alpha_i}\in{\mathfrak g}_{\alpha-\alpha_i}$ and $\alpha_i$ can be any simple root.
By an obvious induction hypothesis, we may assume that ${\mathfrak k}_{(\alpha-\alpha_i)}\subset\pi({\mathcal L}_{(\alpha-\alpha_i)})$ and ${\mathfrak k}_{(\alpha-2\alpha_i)}\subset\pi({\mathcal L}_{(\alpha-2\alpha_i)})$.
Then $y_\alpha+\omega(y_\alpha) = [e_i-f_i,y_{\alpha-\alpha_i}+\omega(y_{\alpha-\alpha_i})]+[f_i,y_{\alpha-\alpha_i}]+\omega([f_i,y_{\alpha-\alpha_i}])$.
Since $[e_i-f_i,y_{\alpha-\alpha_i}+\omega(y_{\alpha-\alpha_i})]\in\pi([x_i,{\mathcal L}_{(\alpha-\alpha_i)}])$ and $[f_i,y_{\alpha-\alpha_i}]+\omega([f_i,y_{\alpha-\alpha_i}])\in\pi({\mathcal L}_{(\alpha-2\alpha_i)})$, it follows that $y_\alpha+\omega(y_\alpha)\in\pi({\mathcal L}_{(\alpha)})$.
For injectivity, we remark that the inequalities $$\dim{\mathfrak g}_\alpha\geq \dim{\mathcal L}_{(\alpha)}/{\mathcal L}_{<\alpha}\geq\dim{\mathfrak k}_{(\alpha)}/{\mathfrak k}_{<\alpha}=\dim{\mathfrak g}_\alpha$$
establish that ${\rm ker}\,\pi\cap{\mathcal L}_{(\alpha)}=\{0\}$.
\end{proof}

\begin{remark}
Suppose $A=\begin{pmatrix} 2 & -r \\ -s & 2\end{pmatrix}$ where $r,s\neq 0$.
It is easy to see that if we quotient ${\mathfrak k}$ by the ideal generated by $[X_1,[X_1,X_2]]+r^2 X_2$ and $[X_2,[X_2,X_1]]+s^2X_1$ then we obtain an epimorphism ${\mathfrak k}\rightarrow\mathfrak{so}_3$.
This corresponds to repeatedly applying Construction~\ref{reductionconstruction}(a) below to the complete bipartite graph to obtain a diagram of type $A_2$.
\end{remark}

In what follows, we suppose that the generalized Cartan matrix $A$ is indecomposable.
Then there is a well-defined, unique up to scalar multiplication {\it length function} $| \cdot  |$ on the simple roots such that $\frac{a_{ij}}{a_{ji}}=\frac{|\alpha_j|^2}{|\alpha_i|^2}$ whenever $a_{ij}\neq 0$.
After scaling we may assume that $|\alpha_i|^2\in{\mathbb N}$ for any $i$, and that the square lengths $|\alpha_i|^2$ have no common factor.

\begin{definition}
A {\bf simply laced cover diagram of ${\mathfrak g}$} (or just a {\bf cover diagram} for short) is a simply laced diagram $D$ with $n_i$ vertices $\alpha_i^{(1)}$, \ldots, $\alpha_i^{(n_i)}$ for each simple root $\alpha_i$ of ${\mathfrak g}$ (where $n_i$ are some positive integers), and such that each $\alpha_i^{(k)}$ is connected to exactly $|a_{ij}|$ of the vertices $\alpha_j^{(l)}$ for $j\neq i$ and to none of the other vertices $\alpha_i^{(l)}$.
\end{definition}

We remark that the $n_i$ are related by the formula $\frac{n_i}{n_j}=\frac{a_{ij}}{a_{ji}}$ whenever $a_{ij}\neq 0$, hence $n_i=\frac{M}{|\alpha_i|^2}$ for some constant $M$.
It follows that $M$ is divisible by all $|\alpha_i|^2$.
Moreover, each $n_i$ must be divisible by any non-zero value $|a_{ij}|$, so that $M$ is divisible by ${\rm lcm}_{j\neq k: a_{jk}\neq 0}(
|\alpha_j|^2 \cdot|a_{jk}|)$.
In the special case that $M={\rm lcm}_{j\neq k:a_{jk}\neq 0}(|\alpha_j|^2\cdot|a_{jk}|)$ we call the diagram to be of {\bf minimal rank}.

Clearly, one can construct a minimal rank simply laced cover diagram for ${\mathfrak g}$ by setting 
$$n_i=\frac{{\rm lcm}_{j\neq k:a_{jk}\neq 0}(|\alpha_j|^2\cdot|a_{jk}|)}{|\alpha_i^2|}$$ for all $i$ and for each pair $(i,j)$ with $a_{ij}<0$, arbitrarily dividing the vertices $\alpha_i^{(1)}$, \ldots, $\alpha_i^{(n_i)}$ (resp.\ $\alpha_j^{(1)}$, \ldots, $\alpha_j^{(n_j)}$) into $m=\frac{n_i}{|a_{ij}|}=\frac{n_j}{|a_{ji}|}$ subsets $S_1$, \ldots, $S_m$ (resp.\ $S'_1$, \ldots, $S'_m$) of $|a_{ij}|$ (resp.\ $|a_{ji}|$) vertices with every vertex in $S_k$ joined to every vertex in $S'_k$.

As the following examples show, not every connected cover diagram is minimal rank, and two minimal rank cover diagrams need not be isomorphic.

\begin{example}
\begin{enumerate}
\item The Kac--Moody algebra which has generalized Cartan matrix $\begin{pmatrix} 2 & -1 & -1 \\ -2 & 2 & -2 \\ -2 & -2 & 2 \end{pmatrix}$ has (at least) the following two simply laced cover diagrams:
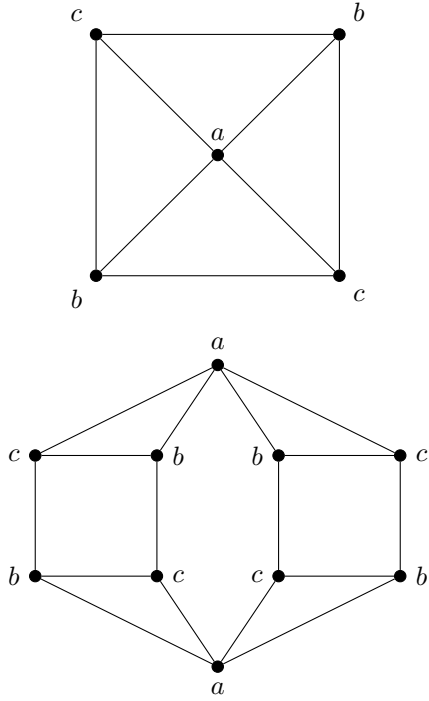
\begin{center}
\begin{tikzpicture}[scale=0.8]
\begin{scope}
    \node[dnode,label=below left:$b$] (b1) at (-2,-2) {};
    \node[dnode,label=above right:$b$] (b2) at (2,2) {};
    \node[dnode,label=above:$a$] (a) at (0,0) {};
    \node[dnode,label=below right:$c$] (c1) at (2,-2) {};
    \node[dnode,label=above left:$c$] (c2) at (-2,2) {};

    \draw[sedge] (b1) -- (c1) -- (b2) -- (c2) -- (b1);
    \draw[sedge] (b1) -- (a) -- (b2);
    \draw[sedge] (c1) -- (a) -- (c2);
\end{scope}
\end{tikzpicture}

\begin{tikzpicture}[scale=0.8]
\begin{scope}
    \node[dnode,label=below:$a$] (a1) at (0,-2.5) {};
    \node[dnode,label=above:$a$] (a2) at (0,2.5) {};

    \node[dnode,label=left:$b$] (b1) at (-3,-1) {};
    \node[dnode,label=right:$b$] (b2) at (-1,1) {};
    \node[dnode,label=left:$b$] (b3) at (1,1) {};
    \node[dnode,label=right:$b$] (b4) at (3,-1) {};

    \node[dnode,label=left:$c$] (c1) at (-3,1) {};
    \node[dnode,label=right:$c$] (c2) at (-1,-1) {};
    \node[dnode,label=left:$c$] (c3) at (1,-1) {};
    \node[dnode,label=right:$c$] (c4) at (3,1) {};

    \draw[sedge] (a1) -- (b1) -- (c1) --
                 (a2) -- (b2) -- (c2) --
                 (a1) -- (b4) -- (c4) --
                 (a2) -- (b3) -- (c3) --
                 (a1);
    \draw[sedge] (b1) -- (c2);
    \draw[sedge] (b2) -- (c1);
    \draw[sedge] (b3) -- (c4);
    \draw[sedge] (b4) -- (c3);
\end{scope}
\end{tikzpicture}
%
%
%
%
%
%
\end{center}
\item If ${\mathfrak g}$ has symmetrizable Cartan matrix $\begin{pmatrix} 2 & -3 & -6 \\ -5 & 2 & -5 \\ -2 & -1 & 2\end{pmatrix}$, then under the assumptions above we have $|\alpha_1|^2=5$, $|\alpha_2|^2=3$ and $|\alpha_3|^2=15$.
Thus ${\rm lcm}_{j\neq k:a_{jk}\neq 0}(|\alpha_j|^2\cdot|a_{jk}|)=30$ and therefore $n_1=6$, $n_2=10$, $n_3=2$.
Note that $\alpha_3^{(1)}$ and $\alpha_3^{(2)}$ are connected to all of the vertices $\alpha_1^{(1)}$, \ldots, $\alpha_1^{(6)}$, but each to only half of $\alpha_2^{(1)}$, \ldots, $\alpha_2^{(10)}$.
Similarly, the vertices $\alpha_2^{(i)}$ also divide into two groups of five, each connecting to three of the vertices $\alpha_1^{(1)}$, \ldots, $\alpha_1^{(6)}$.
After renumbering we may assume that $\alpha_1^{(1)}$, $\alpha_1^{(2)}$, $\alpha_1^{(3)}$ are connected to all of $\alpha_2^{(1)}$, \ldots, $\alpha_2^{(5)}$.
It is not hard to see that there are three isomorphism classes of minimal rank cover diagrams for ${\mathfrak g}$, given by diagrams in which $\alpha_3^{(1)}$ connects to $0$, $1$ or $2$ of the vertices $\alpha_2^{(1)}$, \ldots, $\alpha_2^{(5)}$.
\end{enumerate}
\end{example}

\begin{remark}
If ${\mathfrak g}$ is of finite (resp.\ affine) type then there is a unique choice of connected simply laced cover diagram for ${\mathfrak g}$, which is also finite (resp.\ affine).
Specifically, for the finite type Lie algebras of type $B_n$, $C_n$, $F_4$ and $G_2$ one obtains simply laced cover diagrams of type $D_{n+1}$, $A_{2n-1}$, $E_6$ and $D_4$, and similarly for the corresponding (untwisted) affine types.
The twisted affine types all have simply laced cover diagrams which are of affine type $D$ except for the dual of affine $F_4$, which has simply laced cover ${E}^+_7$.
If ${\mathfrak g}$ is an arbitrary Kac--Moody Lie algebra of rank two then there exists a unique choice of simply laced cover diagram, constructed in Section~\ref{rank2KM}.
\end{remark}

If the generalized Cartan matrix of ${\mathfrak g}$ is not indecomposable then a minimal rank simply laced cover diagram for ${\mathfrak g}$ is one which has the smallest possible number of vertices.
Such a diagram can be constructed as the union of the (minimal rank) simply laced cover diagrams for the simple summands of ${\mathfrak g}$.

Let ${\mathfrak g}$ be an arbitrary symmetrizable Kac--Moody algebra and let $\tilde{\mathfrak g}$ be the Kac--Moody algebra associated to some simply laced cover diagram for ${\mathfrak g}$.
Let $e_i^{(k)}$, $f_i^{(k)}$, $h_i^{(k)}$ be the simple root elements corresponding to the vertex $\alpha_i^{(k)}$, for $1\leq k\leq n_i$.
As in the rank 2 case there is a natural embedding $\tilde\varphi:{\mathfrak g}\rightarrow\tilde{\mathfrak g}$ which sends $e_i$ (resp.\ $f_i$) to $\sum_{k=1}^{n_i} e_i^{(k)}$ (resp.\ $\sum_{k=1}^{n_i} f_i^{(k)}$) and which induces a map from the extended Weyl group of ${\mathfrak g}$ to that of $\tilde{\mathfrak g}$.
Clearly, there is also a corresponding embedding ${\mathfrak k}\hookrightarrow\tilde{\mathfrak k}$.

\section{Some algebraic properties of $\mathfrak{k}$}
In this section we collect some consequences of Berman's presentation of the maximal compact subalgebra of a Kac--Moody algebra.

\subsection{Automorphisms}
For $i=1,\ldots, n$ let $\epsilon_i \in \{\pm 1\}$. Then there is an automorphism $\varphi_\epsilon$ of $\mathfrak k$ characterized by $\varphi(X_i)=\epsilon_i X_i$, called a \textbf{sign automorphism}.

If $\pi\in\mathrm{Sym}(n)$ is a permutation which preserves the generalized Cartan matrix of ${\mathfrak g}$ (i.e., $a_{\pi(i)\pi(j)}=a_{ij}$ for all $i$, $j$) then 
there is an induced automorphism $\varphi_\pi$ of $\mathfrak k$ satisfying 
$\varphi_\pi(X_i)=X_{\pi(i)}$. Such an automorphism is called a \textbf{graph automorphism}.
(In the simply-laced case $\pi$ corresponds exactly to an automorphism of the diagram of ${\mathfrak g}$, i.e., a permutation of the vertices which preserves adjacency.)

\begin{lemma}\label{comp} Let $\mathfrak g$ be a Kac--Moody algebra over a field $k$ of characteristic 0.
\begin{enumerate} 
\item
For $i=1, \ldots, n$, the element $s_i^* \in W^*$ commutes with $\omega$.
\item Every $w \in W^*$ induces an automorphism $\pi(w)$ of $\mathfrak{k}$.
\item If the Kac--Moody algebra $\mathfrak g$ is simply laced, the automorphism $\pi(s_i^*)$ induced by $s_i^*$ via the isomorphism given in Theorem~\ref{Berman} satisfies
\begin{eqnarray*}
X_i & \mapsto & X_i, \\
X_j & \mapsto&  X_j, \text{ if $(i,j) \not \in E$, and} \\
X_j & \mapsto & [X_i,X_j], \text{  if $(i,j) \in E$}. 
\end{eqnarray*}
\end{enumerate}
\end{lemma}

\begin{proof}
Statement (a) has been proved in Remark~\ref{weylgroupremark}.
By (a), each $s_i^*$ stabilizes $\mathfrak k$. Statement (b) therefore follows immediately from \cite[Lem.~3.8(b)]{kac1994infinite}.

Concerning (c), a calculation in $\mathfrak{sl}_2(k)$ shows that $s_i^*(e_i)=-f_i$.
A calculation in $\mathfrak{sl}_3(k)$ shows $s_i^*(e_j)=[e_i,e_j]$, if $(i,j) \in E$, and a calculation in $\mathfrak{sl}_2(k) \oplus \mathfrak{sl}_2(k)$ shows $s_i^*(e_j) = e_j$, if $(i,j) \not\in E$. More calculations --- or use of assertion (a) --- show, furthermore, $s_i^*(f_i)=-e_i$ and $s_i^*(f_j)=-[f_i,f_j]$, if $(i,j) \in E$, and $s_i^*(f_j) = f_j$, if $(i,j) \not\in E$.
In particular, $$s_i^*(e_j-f_j) = s_i^*(e_j)-s_i^*(f_j) = [e_i,e_j]+[f_i,f_j]=[e_i-f_i,e_j-f_j].$$ Statement (c) follows.
\end{proof}

For $w \in W^*$, the induced automorphism $\pi(w) \in \Aut \mathfrak k$ is called a \textbf{Weyl group automorphism}. 


\begin{remark}
\begin{enumerate} 
\item Let $\varphi_+ :  \mathfrak n_+ \to \mathfrak k : x \mapsto x+\omega(x)$ denote the canonical $k$-linear bijection 
(cf.~\cite[p.~3169]{Berman}), and write $\mathfrak k_\alpha:=\varphi_+(\mathfrak g_\alpha)$. Observe that for the analogous $k$-linear bijection $\varphi_- :  \mathfrak n_- \to \mathfrak k : x \mapsto x+\omega(x)$ one has $\mathfrak{k}_\alpha=\varphi_+(\mathfrak{g}_\alpha) = \varphi_-(\mathfrak{g}_{-\alpha})=\mathfrak{k}_{-\alpha}$. 

It follows from Lemma~\ref{comp}(a) that $\pi(s)(\mathfrak k_\alpha)=\mathfrak k_{s \cdot \alpha}$. Hence, by induction and by the definition of the set of real roots, for any positive real root $\alpha \in \Delta_+$ there is a Weyl group automorphism $\pi(w)$ and a positive simple root $\alpha_i$ such that $\pi(w)(\mathfrak k_\alpha)=\mathfrak k_{\alpha_i}=k X_i$.
 
\item The set of subspaces $\{\mathfrak k_\gamma \mid \gamma \in \Delta^{\mathrm{re}} \cap \Delta_+\}$ is invariant under the action of 
the group of Weyl group automorphisms. It can be identified with the walls of the Coxeter complex of the Weyl group $W$. (Cf.\ \cite[Rem.~3.8]{kac1994infinite}.)
\end{enumerate}
\end{remark}

\begin{remark} 
If ${\mathfrak g}$ is simply laced then for $i$, $j$ in the same connected component of the diagram of $\mathfrak{k}$ there is an automorphism 
such that $\varphi(X_i)=X_j$. This is because, if $(i,j)$ is an edge, then $$\pi(s_i^*s_j^*)(X_i)\stackrel{\ref{comp}}{=}\pi(s_i^*)([X_j,X_i])=[\pi(s_i^*)(X_j),\pi(s_i^*)(X_i)]\stackrel{\ref{comp}}{=}[[X_i,X_j],X_i]\stackrel{\ref{Berman}}{=}X_j;$$ thus, the claim 
follows by 
induction.

This can be used as follows: Let $\mathfrak k$ be the maximal compact subalgebra of a Kac--Moody algebra of type $AE_4$ (see Section~\ref{DD}). Then the generator $X_4$ is contained in a subalgebra isomorphic to the maximal compact subalgebra of a Kac--Moody algebra of type $A_2^+$. Indeed, let $\varphi$ be a Weyl group automorphism such that 
$\varphi(X_3)=X_4$. Then $\varphi(\mathfrak \langle X_1,X_2,X_3 \rangle)$ is as required, as by Theorem~\ref{Berman} the Lie algebra $\langle X_1, X_2, X_3 \rangle$ equals the maximal compact subalgebra of the Kac--Moody algebra with positive simple roots $\alpha_1$, $\alpha_2$, $\alpha_3$.
\end{remark}

\subsection{A contraction of $\mathfrak k$.}
Let $\lieg$ be a symmetrizable Kac--Moody algebra over $\RR$
with Chevalley generators $e_i$, $f_i$, $h_i$, $i=1, \ldots, n$. 
For $\epsilon >0$ define $\omega_\epsilon$ to be the Lie algebra automorphism satisfying
$$ \omega_\epsilon (e_i)= -\epsilon f_i,~ \omega_\epsilon(f_i)=-\frac{1}{\epsilon} e_i,~ \omega_\epsilon(h_i)=-h_i;$$
moreover, set $\mathfrak k _\epsilon:=\Fix \omega_\epsilon$.
Observe that $\mathfrak k=\mathfrak k_1$ and that $X_i^\epsilon:=e_i-\epsilon f_i \in \mathfrak k_\epsilon$ for $i=1,\ldots, n$.
Moreover, the automorphism $\theta_\epsilon$ of ${\mathfrak g}$ given by $e_i\mapsto\frac{1}{\sqrt\epsilon}e_i$ and $f_i\mapsto\sqrt\epsilon f_i$ for all $i$ satisfies $$\theta_\epsilon(X_i)=\frac{1}{\sqrt\epsilon}X_i^\epsilon, \quad \quad \omega_\epsilon=\theta^2_\epsilon\circ\omega=\theta_\epsilon\circ\omega\circ\theta_\epsilon^{-1}.$$
Thus $\theta_\epsilon$ maps ${\mathfrak k}$ isomorphically onto ${\mathfrak k}_\epsilon$.
By applying $\theta_\epsilon$ to $P_{-a_{ij}}(\ad X_i)(X_j)$ (using the notation of Theorem~\ref{Berman2}), we obtain the relations:
$$P_{-a_{ij}}^\epsilon(\ad X_i^\epsilon)(X_j^\epsilon) = 0\;\;\;\mbox{where}\;\; P_{m}^\epsilon(t)=\epsilon^{\frac{m+1}{2}}P_m\left(\frac{t}{\sqrt\epsilon}\right)$$
that is, $P_m^\epsilon(t)=(t^2+m^2\epsilon)\cdots (t^2+\epsilon)$ for $m$ odd, and $P_m^\epsilon(t)=(t^2+m^2\epsilon)\cdots (t^2+4\epsilon)t$ for $m$ even.
In particular, $[X_i^\epsilon,[X_i^\epsilon,X_j^\epsilon]]  =  -\epsilon 
X_j^\epsilon$, if $a_{ij}=-1$.

Since $\theta_\epsilon$ maps ${\mathfrak k}$ isomorphically onto ${\mathfrak k}_\epsilon$, we have:

\begin{proposition}\label{contractionprop} The subalgebra $\mathfrak k_\epsilon$ is isomorphic to the quotient of the free Lie algebra over $k$
generated by $X_1, \ldots, X_n$ subject to the relations 
$$
P_{-a_{ij}}^\epsilon(\ad X_i)(X_j) =0
$$
via the map $X_i \mapsto e_i-\epsilon f_i$.
\end{proposition}

Note that, if we set $\epsilon=0$ in the above presentation, the resulting algebra is isomorphic to $\mathfrak n_+$ by the Gabber--Kac Theorem \cite[Thm.~9.11]{kac1994infinite}.
This means that $\mathfrak n_+$
is a \textbf{contraction} of the maximal compact subalgebra $\mathfrak k=\mathfrak k_1$ in the sense of \cite{FialowskideMontigny}.

\subsection{Quotients}
Let $k$ be a field of characteristic 0 and $\mathfrak g$ a Kac--Moody algebra over $k$ with simply laced diagram $D$. Due to the Coxeter-like presentation of the maximal compact subalgebra $\mathfrak k$ it is possible to exhibit quotients of $\mathfrak k$ if $D$ has a certain shape.

For a graph $D$, let $\mathfrak k(D)$ denote the maximal compact subalgebra of the Kac--Moody algebra $\mathfrak g$ over $k$ with diagram $D$.

\begin{construction}\label{reductionconstruction}
Suppose that there are distinct vertices $v_i$, $v_j$ of the diagram $D$ such that any vertex $v_r$ distinct from $v_i$, $v_j$ is connected to $v_i$ if and only if $v_r$ is connected to $v_j$. 

\begin{enumerate}
\item If $v_i$ and $v_j$ are not connected by an edge, let $D'$ be the diagram obtained from $D$ by deleting the vertex $v_j$. 
Let $\mathfrak k':=\mathfrak k(D')$ and $X_1', \ldots, X_n'$ its Berman generators.  
Then there is a well-defined epimorphism of Lie algebras
$\varphi\colon \mathfrak k \to \mathfrak k'$
determined by $\varphi(X_r):=X_r'$ for $r \neq j$ and $\varphi(X_j):=X_i'$.
\item If $v_i$ and $v_j$ are connected by an edge, let $D'$ be the diagram obtained from $D$ by deleting all edges emanating from $v_j$ except for the edge $(v_i,v_j)$. As above, let $\mathfrak k':=\mathfrak k(D')$ and $X_1', \ldots, X_n'$ its Berman generators. 
 Then there is a well-defined epimorphism of Lie algebras
$\varphi\colon \mathfrak k \to \mathfrak k'$
determined by $\varphi(X_r):=X_r'$ for $r \neq j$ and $\varphi(X_j):=[X_i',X_j']$.

This can be checked by using the Weyl automorphisms introduced in Lemma~\ref{comp}. For instance, for all $r \neq i, j$ with $(v_r,v_j) \in E_D$ (which is equivalent to $(v_r,v_i) \in E_D$), one has
\begin{eqnarray*}
[\phi(X_j), [\phi(X_j),\phi(X_r)]] & = & [[X'_i,X'_j], [[X'_i,X'_j],X'_r]] \\ & \stackrel{\ref{comp}}{=} & [-\pi(s^*_j)(X'_i),[-\pi(s^*_j)(X'_i),\pi(s_j^*)(X'_r)]] \\ & = & \pi(s^*_j)[X'_i,[X'_i,X'_r]] \\ & \stackrel{\ref{Berman}}{=} & \pi(s^*_j)(-X'_r) \\ & \stackrel{\ref{comp}}{=} & -X'_r \\ & = & \phi(-X_r) \\ & \stackrel{\ref{Berman}}{=} & \phi[X_j,[X_j,X_r]].
\end{eqnarray*}

Case (a) (resp.\ (b)) of Construction \ref{reductionconstruction} corresponds to quotienting ${\mathfrak k}$ by the ideal generated by $(X_i-X_j)$ (resp.\ by all terms of the form $[X_r,[X_i,X_j]]$ where $r\neq i,j$).

\end{enumerate}
\end{construction}

\begin{example}
\begin{enumerate}
\item 
The preceding discussion gives a sequence of epimorphisms of real Lie algebras
$\mathfrak k(D_4^+) \twoheadrightarrow \mathfrak k(D_4) \twoheadrightarrow \mathfrak 
k(A_3)=\mathfrak{so}_4(\mathbb{R}) 
\twoheadrightarrow \mathfrak k(A_2)=\mathfrak{so}_3(\mathbb{R})$.

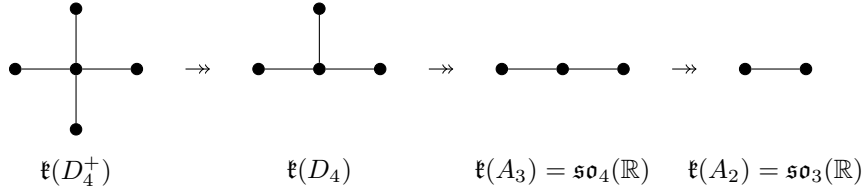
\begin{center}
\begin{tikzpicture}[scale=.8]
\begin{scope}
    \draw (1,-1.7) node  {$\mathfrak k(D_4^+)$};
    \node[dnode] (1) at (0,0) {};
    \node[dnode] (2) at (1,0) {};
    \node[dnode] (3) at (2,0) {};
    \node[dnode] (4) at (1,1) {};
    \node[dnode] (5) at (1,-1) {};
	\draw[sedge] (1) -- (2) -- (3);
	\draw[sedge] (4) -- (2) -- (5);
	\draw[->>] (2.8,0) -- (3.2,0);
\end{scope}
\begin{scope}[shift={(4,0)}]
    \draw (1,-1.7) node  {$\mathfrak k(D_4)$};
    \node[dnode] (1) at (0,0) {};
    \node[dnode] (2) at (1,0) {};
    \node[dnode] (3) at (2,0) {};
    \node[dnode] (4) at (1,1) {};
	\draw[sedge] (1) -- (2) -- (3);
	\draw[sedge] (4) -- (2);
	\draw[->>] (2.8,0) -- (3.2,0);
\end{scope}
\begin{scope}[shift={(8,0)}]
    \draw (1,-1.7) node  {$\mathfrak k(A_3)=\mathfrak{so}_4(\mathbb{R})$};
    \node[dnode] (1) at (0,0) {};
    \node[dnode] (2) at (1,0) {};
    \node[dnode] (3) at (2,0) {};
	\draw[sedge] (1) -- (2) -- (3);
	\draw[->>] (2.8,0) -- (3.2,0);
\end{scope}
\begin{scope}[shift={(12,0)}]
    \draw (0.5,-1.7) node  {$\mathfrak k(A_2)=\mathfrak{so}_3(\mathbb{R})$};
    \node[dnode] (1) at (0,0) {};
    \node[dnode] (2) at (1,0) {};
	\draw[sedge] (1) -- (2);
\end{scope}
\end{tikzpicture}
\end{center}
This sequence can be extended further: Let $\Gamma_n=(\{1, \ldots, 
n\},\{(1,k) \mid 2 \leq k \leq n\})$ denote the star diagram on $n$ vertices and let $\mathfrak k_n$ denote the maximal compact subalgebra of the 
Kac--Moody 
algebra $\mathfrak g_n$ with Dynkin diagram $\Gamma_n$. Then there are epimorphisms $\mathfrak k_n \to \mathfrak k_{n-1}$.  
\item Denoting by $K_4$ the complete graph on four vertices, there similarly is a sequence of epimorphisms $\mathfrak k(K_4) \twoheadrightarrow \mathfrak k(AE_4) \twoheadrightarrow \mathfrak k(A_4)$. 
\end{enumerate}

\end{example}

\section{Generalized spin representations}\label{GSR}

\subsection{Generalized spin representations of $\mathfrak{k}(E_{10}(\RR))$}Let us recall the extension of the spin representation of $\mathfrak k(\liesl_{10}(\RR))$ to 
$\mathfrak k(E_{10})(\RR)$ as described by \cite{DamourKleinschmidtNicolai}, \cite{deBuylHenneauxPaulot} (also 
\cite{Keurentjes}).

\begin{example} \label{theexample}
Let $V$ be a $k$-vector space and $q\colon V \to k$ a quadratic form with associated bilinear form $b$. Then the \textbf{Clifford algebra}
$C:=C(V,q)$ is defined as $C:=T(V)/\langle vw+wv-2b(v,w) \rangle$ where $T(V)$ is the tensor algebra of $V$.

Now let $V=\RR^{10}$ with standard basis vectors $v_i$, let $q=x_1^2+\cdots+x_{10}^2$ and let $C=C(V,q)$. Then in $C$ we have 
$$v_i^2=1 \text{ and } v_iv_j=-v_jv_i.$$ 
Since $C$ is an associative algebra, it becomes a Lie algebra by setting $[A,B]:=AB-BA$. 
 Let the diagram of $\mathfrak{g}(E_{10})(\RR)$ be labelled as 
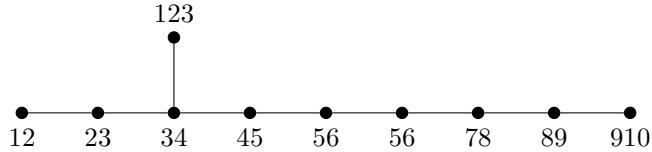
\begin{center}
\begin{tikzpicture}
    \node[dnode,label=below:12] (1) at (0,0) {};
    \node[dnode,label=above:123] (2) at (2,1) {};
    \node[dnode,label=below:23] (3) at (1,0) {};
    \node[dnode,label=below:34] (4) at (2,0) {};
    \node[dnode,label=below:45] (5) at (3,0) {};
    \node[dnode,label=below:56] (6) at (4,0) {};
    \node[dnode,label=below:56] (7) at (5,0) {};
    \node[dnode,label=below:78] (8) at (6,0) {};
    \node[dnode,label=below:89] (9) at (7,0) {};
    \node[dnode,label=below:910] (10) at (8,0) {};

	\draw[sedge] (2) -- (4);
	\draw[sedge] (1) -- (3) -- (4) -- (5) -- (6) -- (7) -- (8) -- (9) -- (10);
\end{tikzpicture}
\end{center}
and define a Lie algebra homomorphism $\rho : \mathfrak{k} \to C$ using these labels, i.e., via
\begin{eqnarray*}
X_1  \mapsto  \frac{1}{2}v_1v_2, & X_2  \mapsto  \frac{1}{2}v_1v_2v_3, & X_3  \mapsto  \frac{1}{2}v_2v_3, \\
X_4  \mapsto  \frac{1}{2}v_3v_4, & X_5 \mapsto  \frac{1}{2}v_4v_5, & X_6  \mapsto  \frac{1}{2}v_5v_6, \\
X_7 \mapsto \frac{1}{2}v_6v_7, & X_8 \mapsto \frac{1}{2}v_7v_8, & X_9  \mapsto  \frac{1}{2}v_8v_9, \\
& X_{10} \mapsto \frac{1}{2}v_9v_{10},
\end{eqnarray*} 
where $X_i$ denotes the Berman generator corresponding to the root $\alpha_i$, enumerated in Bourbaki style as in Section~\ref{DD}.
Observe that each $A_i:=\rho (X_i)$ satisfies $A_i^2=-\frac1 4 \id$. Here we would like to remark that $(v_1v_2v_3)^2 = (v_2v_3)^2 = -1$ depends on $v_i^2 = 1$; for parity reasons, this would not be true in the Clifford algebra $C(V,-q)$, as then  $(v_1v_2v_3)^2 = -(v_2v_3)^2 = 1$.  

Using the criterion established in Remark~\ref{characterization} below, one checks easily that $\rho$ indeed is a Lie algebra homomorphism, i.e., that the 
defining relations of $\mathfrak{k}$ from Theorem~\ref{Berman} are respected.
Indeed, one just needs to establish 
\begin{itemize}
\item[(i)] $A_i^2=-\frac{1}{4}\cdot \id_s$, 
\item[(ii)] $A_iA_j=A_jA_i$, if $(i,j) \not\in E$, 
\item[(iii)] $A_iA_j=-A_jA_i$, if $(i,j) \in E$.
\end{itemize}
 We have already observed (i). Assertions (ii) and (iii) are obvious for $i, j\neq 2$.
Moreover, one quickly computes $(v_1v_2v_3)(v_3v_4) = -(v_3v_4)(v_1v_2v_3)$ and $(v_1v_2v_3)(v_{k_1}v_{k_2}) = (v_{k_1}v_{k_2})(v_1v_2v_3)$, if $\{ k_1, k_2 \}$ is a set of two elements that is either a subset of $\{ 1, 2, 3 \}$ or disjoint from $\{ 1, 2, 3 \}$. Assertions (ii) and (iii) follow.

By \cite[Lemma~20.9]{FultonHarris}, \cite[Proposition~2.4]{Meinrenken:2013} the Clifford algebra $C$ splits over $\CC$ as $C \otimes_\RR \CC \cong \CC^{32 \times 32}$.  
Hence 
$\rho$ affords a 32-dimensional complex representation of $\mathfrak k(E_{10})(\RR)$. The restriction of this representation to the maximal compact subalgebra of the $A_9$-subdiagram, $\mathfrak k(A_9)(\RR)=\so_{10}(\RR)$, coincides with the spin representation of $\so_{10}$ (see e.g.\ \cite[Chapter 20]{FultonHarris}), 
i.e., $\rho$ extends the classical spin representation.

Let $\iota \in \Aut C$ denote the involution (known as parity automorphism) induced by $V \to V : v \mapsto -v$. Let $C_0:=\Fix \iota$ 
and $C_1:=\{w \in C \mid \iota(w)=-w\}$ denote the even and the odd part of $C$. Then $C_0$ and $C_1$ are invariant subspaces
under 
the spin representation of $\so_{10}$ since $\im \rho \subseteq C_0$ (multiplication with a product of the $v_i$ of even length does not change the parity) and these subspaces are irreducible non-isomorphic 
representations of $\mathfrak \so_{10}$ (\cite[Chapter 20]{FultonHarris}). 

The remaining Berman generator $X_2$ of 
$\mathfrak k(E_{10})$ is sent to an element which interchanges $C_0$ and $C_1$. 
\end{example}

\begin{remark}
A calculation shows that $\im \rho$ is the linear span of all elements of the form $v_{i_1}\cdots v_{i_k}$, where $\{ i_1, ..., i_k \} = I\subseteq 
\{1, \ldots, 10\}$ with $|I| \in \{2,3,6,7,10\}$. Therefore, $\dim \im(\rho) = 45 + 120 + 210 + 120 + 1 = 496$. Since $\im(\rho) \leq C \cong \mathbb{R}^{32 \times 32}$ by \cite[Section~2.2.3]{Meinrenken:2013} and since $\im(\rho)$ is compact and semisimple by Theorem~\ref{Mainexistencetheorem}, this dimension $\dim \im(\rho) = 496$ implies 
$\im(\rho) \cong \so_{32}(\RR)$ (see also \cite{damour2006k}).  

The existence of Example~\ref{theexample} is not peculiar to the diagram $E_{10}$, it can be generalized to arbitrary diagrams $E_n$ in the obvious way. A careful analysis of dimensions combined with the Cartan--Bott periodicity of Clifford algebras allows one to determine the isomorphism types of the quotients for the whole $E_n$ series. This is carried out in Appendix~\ref{CartanBott}. A key observation is that the cardinality $|I|$ from above in general has to be equal to $2$ or $3$ modulo $4$ (see Lemma~\ref{lowerbounddimension}).
\end{remark}

\begin{remark} Let $\rho\colon \mathfrak{so}_{10}(\RR) \to \CC^{n \times n}$ be a representation. To extend $\rho$ to a representation of $\mathfrak k(E_{10})$, it suffices to find a matrix $X \in \CC^{n \times n}$ such that for $A_i:=\rho(X_i)$, $1 \leq i \leq 10$, $i \neq 2$, the following equations are satisfied (where we again use the labelling of the diagram $E_{10}$ as given in Section~\ref{DD}):
\begin{eqnarray*}
[A_i,X]        & = &   0 \quad \quad \mbox{for $1 \leq i \leq 10$, $i \neq 2, 4$}, \\
{}[A_4,[A_4,X]] & = & -X, \\
{}[X,[X,A_4]]    & = & -A_4.   
\end{eqnarray*}
Theorem~\ref{Berman} then implies that $\rho$ can be extended to $\mathfrak k(E_{10})$ by setting $\rho(X_2):=X$.

The first two sets of equations define a linear subspace, the third set of equations yields a family of quadratic equations. With the help of a Gr\"obner basis one can compute that in case of the spin representation, this variety is isomorphic to $\CC^\times$, i.e., the extension is unique up to a scalar. 
\end{remark}

\subsection{Generalized spin representations for the simply-laced case}
Throughout this section, let $k$ be a field of characteristic 0, let $\mathfrak g$ be a Kac--Moody algebra over $k$ 
with simply laced diagram
and let $\mathfrak k$ be its maximal compact subalgebra.

Let $L:=k(I)$, where $I$ is a square root of $-1$ and denote by $\id_s \in L^{s \times s}$ the identity matrix. 

\begin{definition}\label{genspinrep}
A representation $\rho\colon \mathfrak k \to \End(L^s)$ is called a \textbf{generalized spin representation} if
the images of the Berman generators from Theorem~\ref{Berman} satisfy  
$$\rho(X_i)^2=-\frac{1}{4}\id_s \text{ for } i=1, \ldots, n.$$
\end{definition}

\begin{remark} \label{characterization}
\begin{enumerate}
\item 
Since $\rho$ is assumed to be a representation, it follows from the defining relations that $\rho(X_i)$ and $\rho(X_j)$ commute if $(i,j) \not \in E$.
On the other hand, if $(i,j) \in E$, then $A:=\rho(X_i)$ and $B:=\rho(X_j)$ anticommute.
Indeed, we have 
$$-B\stackrel{\ref{Berman}}{=}[A,[A,B]]=A^2B-2ABA+BA^2=-\frac 1 2 B-2ABA$$
from which the claim follows after multiplying with $A^{-1}=-4A \Longleftrightarrow A^2 = - \frac{1}{4}\id_s$.
\item Conversely, suppose that there are matrices $A_i \in L^{s \times s}$ satisfying 
\begin{enumerate} 
\item $A_i^2=-\frac{1}{4}\cdot \id_s$,
\item $A_iA_j=A_jA_i$ if $(i,j) \not\in E$,
\item $A_iA_j=-A_jA_i$ if $(i,j) \in E$. 
\end{enumerate}
Then, by reversing the argument in the above computation, the assignment $X_i \mapsto A_i$ gives rise to a representation of $\mathfrak k$. 
\end{enumerate}
\end{remark}

\begin{remark}\label{coxspin} Let $\rho$ be a generalized spin representation of $\mathfrak k$ and set $S_i:=2I\cdot\rho(X_i)$. Let $W$ be a Coxeter group defined by the presentation 
$$W=\langle s_1, \ldots, s_n \mid (s_is_j)^{m_{ij}}=1 \rangle,$$ where $m_{ii}=1$ and $m_{ij}=2$ if $(i,j) \not\in 
E$, while $m_{ij}\in\{ 3, 4 \}$ if $(i,j) \in E$. Then the assignment $s_i \mapsto S_i$ gives a representation of $W$.  
\end{remark}

Write $\mathfrak k_{\leq r}:=\langle X_1, \ldots, X_r \rangle$. 
\begin{theorem}\label{MainThm} Let $1\leq r <n$.
Let $\rho : \mathfrak k_{\leq r} \to \End (L^s)$ be a generalized spin representation.
\begin{enumerate}
\item If $X_{r+1}$ centralizes $\mathfrak k_{\leq r}$, then $\rho$ can be extended to a generalized spin 
representation $\rho' : \mathfrak k_{\leq r+1} \to \End(L^s)$ by setting $\rho'(X_{r+1}):=\frac 1 2 I \cdot \id_s$. 
\item If $X_{r+1}$ does not centralize $\mathfrak k_{\leq r}$, then $\rho$ can be extended to a generalized spin 
representation $\rho'\colon \mathfrak k_{\leq r+1} \to \End(L^s\oplus L^s)$ as follows. Define the sign automorphism $s_0 : \mathfrak k_{\leq r} \to L^s$ via $$s_0(X_i):=\left\{
\begin{array}{cc} 
X_i, & \text{ if } (i,r+1)\not\in E, \\
-X_i, & \text{ if } (i,r+1) \in E,
\end{array} \right.
$$ let
$$\rho'|_{\mathfrak k_{\leq r}}:= \rho \oplus \rho \circ s_0,$$
and $$\rho'(X_{r+1}):=\frac{1}{2}I \cdot \id_s \otimes 
\begin{pmatrix} 0 & 1 \\ 1 & 0 \end{pmatrix}.$$
\end{enumerate}
\end{theorem}

\begin{proof} If $X_{r+1}$ centralizes $\mathfrak k_{\leq r}$, it is clear that $\rho'$ is well-defined and that $\rho'(X_{r+1})^2=-\frac{1}{4} \id_s$. 

In the second case it is clear that $\rho'|_{\mathfrak k_\leq r}$ is a generalized spin representation of $\mathfrak k_{\leq r}$ which extends $\rho$. It is easy to check that 
$\rho'(X_i)$ commutes with $\rho'(X_{r+1})$ if $(i,r+1) \not\in E$, and that $\rho'(X_i)$ anticommutes with 
$\rho'(X_{r+1})$ if $(i,r+1) \in E$. 
Remark~\ref{characterization} therefore implies that $\rho'$ is a generalized spin representation. 
\end{proof}

For a graph $G=(V,E)$, a subset $M \subseteq V$ is called a \textbf{coclique} if the subgraph of $G$ induced on $M$ 
does not contain any edges, i.e., if no two elements $m_1$, $m_2$ in $M$ are 
connected by an edge. 

\begin{corollary} \label{Cor34}
Let $n$ be the cardinality of the diagram of $\lieg$ and let $r$ be the size of a maximal coclique of that diagram. Then there exists a $2^{n-r}$-dimensional generalized spin representation of $\mathfrak k$.
Furthermore, if the diagram is irreducible, then there exists a $2^{n-1}$-dimensional {\em maximal} generalized spin representation of $\mathfrak k$.
\end{corollary}

\begin{proof}
Up to a change of labelling the set $M:=\{\alpha_1,\ldots, \alpha_r\}$ forms a maximal coclique. 
The map $\rho : \mathfrak k_{\leq r}\to \End(L^1) : X_i \mapsto \frac 1 2 I \cdot \id_1$ is a generalized spin 
representation. By Theorem~\ref{MainThm}, the representation $\rho$ can be extended inductively to a generalized spin representation of 
$\mathfrak{k}$; the dimension doubles at each step because $M$ was assumed to be a maximal coclique.

For the second claim it suffices to order the vertices of the diagram in such a way that two consecutive vertices are adjacent.
\end{proof}

\begin{remark} An inductive construction of the basic spin representations of the symmetric group similar to the one in Theorem~\ref{MainThm} has 
independently 
been obtained by Maas \cite{Maas}. It is likely that by a combination of the methods of \cite{Maas} and of the present 
article, a 
similar construction of 
generalized 
(basic) spin representations is 
possible for any (simply laced) Coxeter group.
\end{remark}

\subsection{Generalized spin representations for symmetrizable Kac--Moody algebras}
In this section let ${\mathfrak g}$ be an arbitrary symmetrizable Kac--Moody Lie algebra with maximal compact subalgebra ${\mathfrak k}$, and let $n_i$ be the number of vertices associated to the root $\alpha_i$ in a minimal rank simply laced cover diagram for ${\mathfrak g}$.
As above, we assume the ground field $k$ has characteristic zero.

\begin{definition}\label{genspinrep2}
A generalized spin representation for ${\mathfrak k}$ is a Lie algebra homomorphism $\rho:{\mathfrak k}\rightarrow\End(L^s)$ such that each of the Berman generators $X_i$ (see Theorem~\ref{Berman2}) satisfies:
$$\begin{array}{rccl}\left(\rho(X_i)^2+\frac{n_i^2}{4}\id_s\right)\left(\rho(X_i)^2+\frac{(n_i-2)^2}{4}\id_s\right)\ldots \left(\rho(X_i)^2+\id_s\right)\rho(X_i) & = & 0, \\ \mbox{if $n_i$ is even,} \\
\left(\rho(X_i)^2+\frac{n_i^2}{4}\id_s\right)\left(\rho(X_i)^2+\frac{(n_i-2)^2}{4}\id_s\right)\ldots \left(\rho(X_i)^2+\frac{1}{4}\id_s\right) & = & 0, \\ \mbox{if $n_i$ is odd,} \end{array}$$
i.e., $P_{n_i}^{\frac{1}{4}}(\rho(X_i))=0$ (in the notation of Proposition~\ref{contractionprop}).
\end{definition}

Another way of saying this is that $\rho(X_i)$ is semisimple with eigenvalues belonging to the set $\{ \frac{(n_i-2j)}{2}I : 0\leq j\leq n_i\}$.
When the generalized Cartan matrix of ${\mathfrak g}$ is simply laced, this definition clearly coincides with Definition \ref{genspinrep}.

\begin{theorem}\label{Mainexistencetheorem}
Let $L=k(I)$ where $I^2=-1$.
Let ${\mathfrak g}$ be an arbitrary symmetrizable Kac--Moody Lie algebra with maximal compact subalgebra ${\mathfrak k}$.
Then there exists a generalized spin representation $\rho:{\mathfrak k}\rightarrow\End(L^s)$.

Moreover, if $k$ is formally real, then $\rho$ can be considered as a representation ${\mathfrak k}\rightarrow\End(k^{2s})$ with $\im\rho$ compact and, therefore, reductive.
Furthermore, in this case $\im\rho$ is semisimple, if for all $i$ there exists $j\neq i$ such that $a_{ji}$ is odd. Finally, in this case $\mathfrak{k} \cong \ker\rho \oplus \im\rho$. 
\end{theorem}

Note that the condition in the next-to-final sentence of the theorem is satisfied if, for example, ${\mathfrak g}$ has a simply laced diagram which has no isolated nodes. It will follow from the proof that the theorem is actually applicable to all generalized spin representations discussed in Theorem~\ref{MainThm} and Corollary~\ref{Cor34}, in particular the standard generalized spin representation from Example~\ref{theexample}.

\begin{proof}
To see that ${\mathfrak k}$ has a generalized spin representation, let $\tilde{\mathfrak g}$ be the Kac--Moody algebra associated to some minimal rank simply laced cover diagram for ${\mathfrak g}$ and let $\tilde\varphi:{\mathfrak g}\rightarrow\tilde{\mathfrak g}$ be the Lie algebra embedding described in Section~\ref{symmetrizablesec}.
Then it is clear from the earlier discussion that, if $\tilde\rho:\tilde{\mathfrak k}\rightarrow\End(L^s)$ is a generalized spin representation for $\tilde{\mathfrak k}$, then $\rho=\tilde\rho\circ\tilde\varphi|_{\mathfrak k}$ is a generalized spin representation for ${\mathfrak k}$.
(It is, however, not clear that any generalized spin representation for ${\mathfrak k}$ arises in this way.)
Thus the first statement follows immediately from Corollary~\ref{Cor34}.

For the second statement it will suffice to prove that there exists a generalized spin representation $\rho:{\mathfrak k}\rightarrow\End(L^s)$ such that, with respect to an appropriate choice of $k$-basis for $L^s$, each of the images $\rho(X_i)$ is a skew-symmetric $2s\times 2s$ matrix over $k$ and, thus, $\rho$ can be interpreted as a homomorphism ${\mathfrak k}\rightarrow\mathfrak{so}_{2s}(k)$.
Since we can construct generalized spin representations for ${\mathfrak k}$ by restricting from those for the Lie algebra associated to a simply laced cover diagram, it will clearly suffice to show that the representation constructed in Theorem~\ref{MainThm} can be realized by using skew-symmetric matrices only.
For the extension of the representation in part (a) of Theorem~\ref{MainThm} this is obvious, as $L \cong \left\{ \begin{pmatrix} a & b \\ -b & a \end{pmatrix} \mid a, b \in k \right\}$ as $k$-algebras, whence $I$ is represented by the skew-symmetric matix $\begin{pmatrix} 0 & 1 \\ -1 & 0 \end{pmatrix}$.   For the extension of the representation in part (b) of Theorem~\ref{MainThm}, observe that 
$$\begin{pmatrix} 1 & 0 \\ 0 & I\end{pmatrix}\begin{pmatrix} 0 & I \\ I & 0 \end{pmatrix}\begin{pmatrix} 1 & 0 \\ 0 & -I \end{pmatrix} = \begin{pmatrix} 0 & 1 \\ -1 & 0 \end{pmatrix}$$ so that after a change of basis we have instead $\rho'(X_{r+1})=\frac{1}{2}\id_s\otimes\begin{pmatrix} 0 & 1 \\ -1 & 0 \end{pmatrix}$ (while $\rho'|_{{\mathfrak k}_\leq r}$ remains unchanged).
Therefore, if the representation of $\mathfrak{k}_{\leq r}$ consists of skew-symmetric matrices over $k$, one can ensure that the representation of $\mathfrak{k}_{\leq r+1}$ also consists of skew-symmetric matrices over $k$.
Thus $\mathrm{im}(\rho)$ is compact, whence reductive. 

For the statement concerning semisimplicity observe that $\mathfrak{k}$ is perfect.
Indeed, by hypothesis, for each generator $X_i$ of $\mathfrak k$, there is some $j$ such that $a_{ji}$ is odd, and therefore the constant term in the polynomial $P_{-a_{ji}}$ is non-zero.
Since $P_{-a_{ji}}(\ad X_j)(X_i)=0$ by Theorem~\ref{Berman2}, it follows that $X_i$ is contained in the linear span of $(\ad X_j)^{2l}(X_i)$, $l\geq 1$.
Thus, the image $\mathrm{im}(\rho)$ is perfect and, by the above, reductive.
The claim is now obvious, as a perfect direct sum of a semisimple and an abelian Lie algebra necessarily is semisimple.

For the final statement observe that $\mathfrak{k}$ is anisotropic with respect to the invariant bilinear form of the Kac--Moody algebra $\mathfrak{g}$ and so $(\ker\rho)^\perp \cong \im\rho$ is an ideal of $\mathfrak{k}$, where $\perp$ denotes the orthogonality relation with respect to the invariant bilinear form. 
\end{proof}

Let $\mathcal C$ denote the class of all generalized spin representation of $\mathfrak k$. 
We check some closure properties of $\mathcal C$. 

\begin{proposition} \begin{enumerate}
\item $\mathcal C$ is closed under direct sums, quotients, duals and taking subrepresentations.
\item If the generalized Cartan matrix of ${\mathfrak g}$ is simply laced and $\rho_1,\rho_2,\rho_3 \in \mathcal C$, then so is $\rho \colon X_i \mapsto 4 \rho_1(X_i) \otimes \rho_2(X_i) \otimes \rho_3(X_i)$. 
\item More generally, if the generalized Cartan matrix of ${\mathfrak g}$ is simply laced and $\rho_1,\rho_2 \in \mathcal C$, then so is $\rho := 2 I \rho_1 \otimes \rho_2$, where $I$ is a primitive fourth root of unity.
\item If $\rho \in \mathcal C$ and $\varphi$ is either a sign, graph or Weyl group automorphism of $\mathfrak k$, then $\rho \circ \varphi \in \mathcal C$. 
\end{enumerate}
\end{proposition}

\begin{proof}
The first three assertions can be easily verified. 
The fourth assertion is clear if $\varphi$ is a graph or a sign automorphism.
The remaining claim follows from Remark \ref{weylgroupremark}, since if $\rho(X_j)$ has eigenvalues $\frac{rI}{2},\frac{(r-2)I}{2},\ldots, -\frac{rI}{2}$ then so does $\rho(\exp (\xi \ad X_i)(X_j))=\exp(\xi \rho(X_i))(\rho(X_j))$.
\end{proof}

\section{Some Dynkin diagrams} \label{DD}
We give the list of relevant Dynkin diagrams we use in the main text. \\

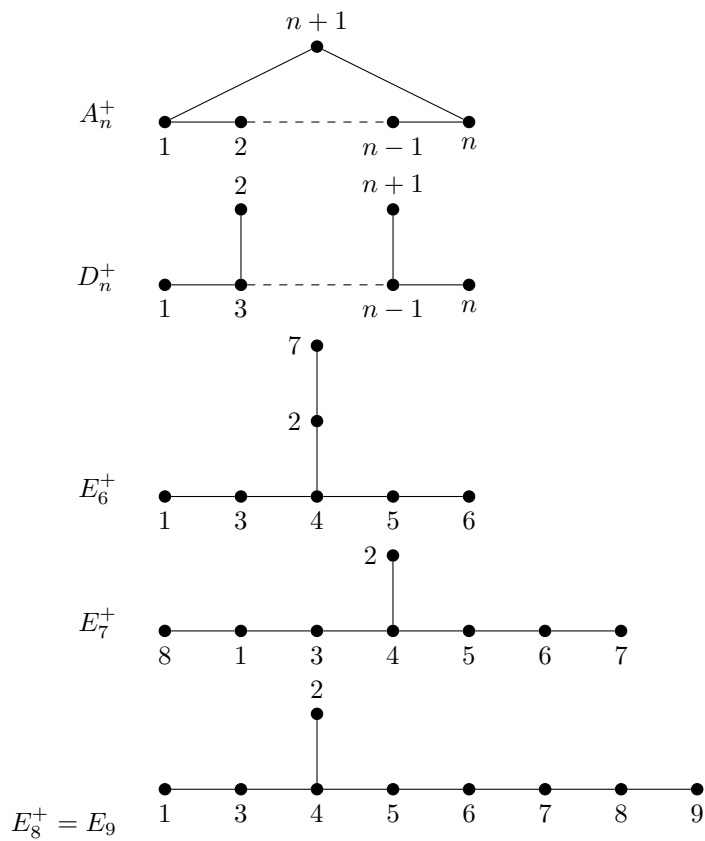
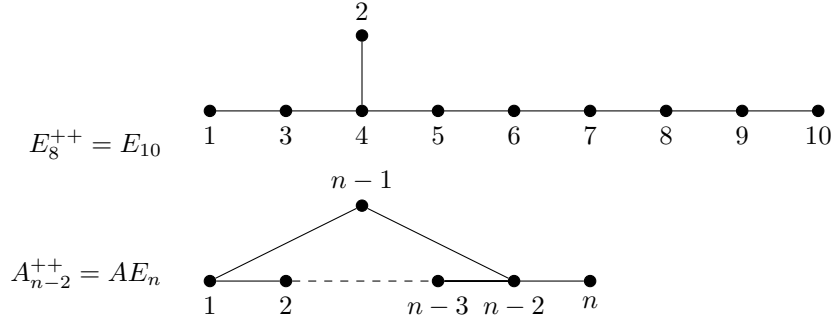
\begin{longtable}{rlrl}
$A_n^+$ & 
\begin{tikzpicture}[baseline]
    \node[dnode,label=below:$1$] (1) at (0,0) {};
    \node[dnode,label=below:$2$] (2) at (1,0) {};
    \node[dnode,label=above:$n+1$] (n+1) at (2,1) {};
    \node[dnode,label=below:$n-1$] (n-1) at (3,0) {};
    \node[dnode,label=below:$n$] (n) at (4,0) {};

	\draw[sedge] (2) -- (1) -- (n+1) -- (n) -- (n-1);
	\draw[sedge,dashed] (2) -- (n-1);
\end{tikzpicture} \\
$D_n^+$ &
\begin{tikzpicture}[baseline]
    \node[dnode,label=below:$1$] (1) at (0,0) {};
    \node[dnode,label=above:$2$] (2) at (1,1) {};
    \node[dnode,label=below:$3$] (3) at (1,0) {};
    \node[dnode,label=below:$n-1$] (n-1) at (3,0) {};
    \node[dnode,label=below:$n$] (n) at (4,0) {};
    \node[dnode,label=above:$n+1$] (n+1) at (3,1) {};

	\draw[sedge] (1) -- (3) -- (2);
	\draw[sedge,dashed] (3) -- (n-1);
	\draw[sedge] (n+1) -- (n-1) -- (n);
\end{tikzpicture} \\
$E_6^+$ &
\begin{tikzpicture}[baseline]
    \node[dnode,label=below:$1$] (1) at (0,0) {};
    \node[dnode,label=left:$2$] (2) at (2,1) {};
    \node[dnode,label=below:$3$] (3) at (1,0) {};
    \node[dnode,label=below:$4$] (4) at (2,0) {};
    \node[dnode,label=below:$5$] (5) at (3,0) {};
    \node[dnode,label=below:$6$] (6) at (4,0) {};
    \node[dnode,label=left:$7$] (7) at (2,2) {};

	\draw[sedge] (4) -- (2) -- (7);
	\draw[sedge] (1) -- (3) -- (4) -- (5) -- (6);
\end{tikzpicture} \\
$E_7^+$ &
\begin{tikzpicture}[baseline]
    \node[dnode,label=below:$1$] (1) at (0,0) {};
    \node[dnode,label=left:$2$] (2) at (2,1) {};
    \node[dnode,label=below:$3$] (3) at (1,0) {};
    \node[dnode,label=below:$4$] (4) at (2,0) {};
    \node[dnode,label=below:$5$] (5) at (3,0) {};
    \node[dnode,label=below:$6$] (6) at (4,0) {};
    \node[dnode,label=below:$7$] (7) at (5,0) {};
    \node[dnode,label=below:$8$] (8) at (-1,0) {};

	\draw[sedge] (4) -- (2);
	\draw[sedge] (8) -- (1) -- (3) -- (4) -- (5) -- (6) -- (7);
\end{tikzpicture} \\
$E_8^+=E_9$ & \multicolumn{3}{l}{\En{9}} \\
$E_8^{++}=E_{10}$ & \multicolumn{3}{l}{\En{10}} \\
$A_{n-2}^{++}=AE_n$ & 
\begin{tikzpicture}[baseline]
    \node[dnode,label=below:$1$] (1) at (0,0) {};
    \node[dnode,label=below:$2$] (2) at (1,0) {};
    \node[dnode,label=above:$n-1$] (n-1) at (2,1) {};
    \node[dnode,label=below:$n-3$] (n-3) at (3,0) {};
    \node[dnode,label=below:$n-2$] (n-2) at (4,0) {};
    \node[dnode,label=below:$n$] (n) at (5,0) {};

	\draw[sedge] (2) -- (1) -- (n-1) -- (n-2) -- (n-3) -- (n);
	\draw[sedge,dashed] (2) -- (n-3);
\end{tikzpicture}
\end{longtable}

\appendix

\renewcommand{\labelenumi}{{\rm(\arabic{enumi})}}
\renewcommand{\labelenumii}{{\rm(\alph{enumii})}}

\bigskip

{\section{Cartan--Bott periodicity for the real $E_n$ series \\ (by Max Horn and Ralf K\"ohl)} \label{CartanBott} }

In this appendix we continue the investigation of the generalized spin representations introduced in the main text. We focus on the $E_n$ series and use the original description of the generalized spin representation from \cite{DamourKleinschmidtNicolai}, \cite{deBuylHenneauxPaulot} via Clifford algebras (see Example~\ref{theexample}).
The $E_n$ series is traditionally only defined for $n\in\{6,7,8\}$. However,
using the Bourbaki style labeling shown in Figure~\ref{fig:En}, it naturally
extends to arbitrary $n\geq 3$. Using this description, one has
$E_3=A_2\oplus A_1$, $E_4=A_4$, $E_5=D_5$
(see Figure~\ref{fig:E3-to-E8}).

\begin{figure}[h]
\begin{tikzpicture}
    \draw (-0.5,0.5) node[anchor=east]  {$E_n$};
    \node[dnode,label=below:1] (1) at (0,0) {};
    \node[dnode,label=above:2] (2) at (2,1) {};
    \node[dnode,label=below:3] (3) at (1,0) {};
    \node[dnode,label=below:4] (4) at (2,0) {};
    \node[dnode,label=below:5] (5) at (3,0) {};
    \node[dnode,label=below:6] (6) at (4,0) {};
    \node[dnode,label=below:$n$] (n) at (5,0) {};

    \path (1) edge[sedge] (3)
          (3) edge[sedge] (4)
          (2) edge[sedge] (4)
          (4) edge[sedge] (5)
          (5) edge[sedge] (6)
          (6) edge[sedge,dashed] (n)
          ;
\end{tikzpicture}
\caption{The Dynkin diagram of type $E_n$}
\label{fig:En}
\end{figure}
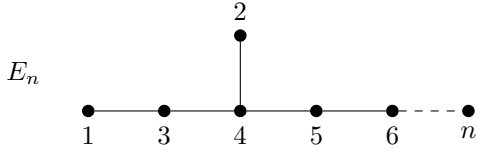

An elementary combinatorial counting argument using binomial coefficients allows us to determine lower bounds for the $\RR$-dimension of the images of the generalized spin representation. These images have to be compact, whence reductive by Theorem~\ref{Mainexistencetheorem} and even semisimple, if the diagram is irreducible. One therefore obtains an upper bound for their $\RR$-dimension via the maximal compact Lie subalgebras of the Clifford algebras. As it turns out, the lower and the upper bounds coincide, providing the following Cartan--Bott periodicity.

\begin{customthm}{A}[Cartan--Bott periodicity of the $E_{n}$ series]\label{thm:A}
Let $n \in \NN$ with $n\geq 4$, let $\mathfrak{k}$ be the maximal compact Lie subalgebra of the split real Kac--Moody Lie algebra of type $E_n$, let $C=C(\RR^n,q)$ be the Clifford algebra with respect to the standard positive definite quadratic form $q$ and let $\rho : \mathfrak{k} \to C$ be the standard generalized spin representation.

Then $\im(\rho)$ is isomorphic to
\begin{enumerate}
\setcounter{enumi}{-1}
\item $\mathfrak{so}(2^{\frac{n}{2}})) \leq \RR\otimes_\RR \M{2^{\frac{n}{2}}}{\RR}$, if $n \equiv 0 \pmod 8$,
\item $\mathfrak{so}(2^{\frac{n-1}{2}}) \oplus \mathfrak{so}(2^{\frac{n-1}{2}})  \leq \left( \RR \oplus \RR \right) \otimes_\RR \M{2^{\frac{n-1}{2}}}{\RR}$, if $n \equiv 1 \pmod 8$,
\item $\mathfrak{so}(2^{\frac{n}{2}}) \leq \M{2}{\RR} \otimes_\RR \M{2^{\frac{n-2}{2}}}{\RR}$, if $n \equiv 2 \pmod 8$,
\item $\mathfrak{su}(2^{\frac{n-1}{2}}) \leq \M{2}{\CC} \otimes_\RR \M{2^{\frac{n-3}{2}}}{\RR}$, if $n \equiv 3 \pmod 8$,
\item $\mathfrak{sp}(2^{\frac{n-2}{2}}) \leq \M{2}{\HH} \otimes_\RR \M{2^{\frac{n-4}{2}}}{\RR}$, if $n \equiv 4 \pmod 8$,
\item $\mathfrak{sp}(2^{\frac{n-3}{2}})\oplus \mathfrak{sp}(2^{\frac{n-3}{2}}) \leq \left( \M{2}{\HH} \oplus \M{2}{\HH} \right) \otimes_\RR \M{2^{\frac{n-5}{2}}}{\RR}$, if $n \equiv 5 \pmod 8$,
\item $\mathfrak{sp}(2^{\frac{n-2}{2}}) \leq \M{4}{\HH} \otimes_\RR \M{2^{\frac{n-6}{2}}}{\RR}$, if $n \equiv 6 \pmod 8$,
\item $\mathfrak{su}(2^{\frac{n-1}{2}}) \leq \M{8}{\CC} \otimes_\RR \M{2^{\frac{n-7}{2}}}{\RR}$, if $n \equiv 7 \pmod 8$,
\end{enumerate}
i.e., $\im(\rho)$ is a semisimple maximal compact Lie subalgebra of $C$.
\end{customthm}

\medskip
Along the way we arrive at a structural explanation for the well-known isomorphism types of the maximal compact Lie subalgebras of the semisimple split real Lie algebras of types 
$E_3=A_2\oplus A_1$, $E_4=A_4$, $E_5=D_5$, $E_6$, $E_7$, $E_8$ (cf., e.g., \cite[p.~518, Table V]{Helgason:1978}).

\begin{customthm}{B} \label{thm:B}
The maximal compact Lie subalgebras of the semisimple split real Lie algebras of types
$A_2\oplus A_1$, $A_4$, $D_5$, $E_6$, $E_7$, $E_8$ are isomorphic to
$\mathfrak{u}(2)$, 
$\mathfrak{sp}(2)\cong\mathfrak{so}(5)$, 
$\mathfrak{sp}(2)\oplus\mathfrak{sp}(2)\cong\mathfrak{so}(5)\oplus\mathfrak{so}(5)$, 
$\mathfrak{sp}(4)$, 
$\mathfrak{su}(8)$, 
$\mathfrak{so}(16)$, 
respectively.
\end{customthm}

\bigskip \noindent
\textbf{Acknowledgements.}
We thank Klaus Metsch for pointing out to us the identity of sums of binomial coefficients in 
Proposition~\ref{binomial} and one of the referees for relaying the elegant version of its proof given in this appendix.
  This research has been partially funded by the EPRSC grant 
EP/H02283X. The second author gratefully acknowledges the hospitality of the IHES at Bures-sur-Yvette and of the Albert Einstein Institute at Golm.

\begin{figure}[h]
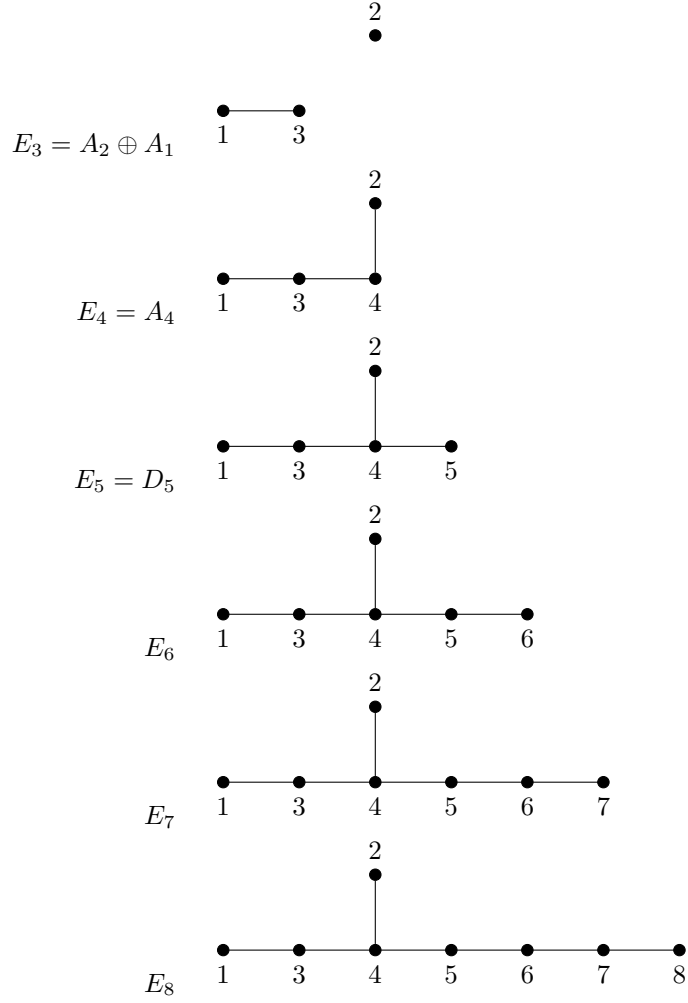

\begin{longtable}{rl}
$E_3=A_2\oplus A_1$ & \En{3} \\ $E_4=A_4$ & \En{4} \\
$E_5=D_5$ & \En{5} \\ $E_6$ & \En{6} \\
$E_7$ & \multicolumn{1}{l}{\En{7}} \\
$E_8$ & \multicolumn{1}{l}{\En{8}} \\
\end{longtable}
\caption{The Dynkin diagrams of types $E_3$ to $E_8$.}
\label{fig:E3-to-E8}
\end{figure}

\subsection{Cartan--Bott periodicity of Clifford algebras} \label{sec:cartan-bott}

Let $\NN = \{1,2,3,\ldots\}$ be the set of natural numbers, and let $\RR$, $\CC$, resp.\ $\HH$ denote the reals, complex numbers resp.\ quaternions. For $n\in\NN$ and a division ring $\mathbb{D}$, denote by $M(n,\mathbb{D})$ the $\mathbb{D}$-algebra of $n\times n$ matrices over $\mathbb{D}$.

Let $V$ be an $\RR$-vector space and $q\colon V \to \RR$ a quadratic form with associated bilinear form $b$. Then the \textbf{Clifford algebra}
$C(V,q)$ is defined as $C(V,q):=T(V)/\langle vw+wv-2b(v,w) \rangle$ where $T(V)$ is the tensor algebra of $V$; cf.\  \cite[Section~4.3]{Kobayashi/Yoshino:2005}, \cite[Chapter 1, \S1]{Lawson/Michelsohn:1989}.

Let $V=\RR^{n}$ with standard basis vectors $v_i$, let $q=x_1^2+\cdots+x_{n}^2$.
Then in $C(V,q)$ we have $v_i^2=1$ and $v_iv_j=-v_jv_i$. 
 
\begin{proposition}[Cartan--Bott periodicity] \label{prop:cartan-bott}
For $n\geq 2$, the Clifford algebra $C(\RR^n,q)$ is isomorphic to the following algebra:
\begin{enumerate}
\setcounter{enumi}{-1}
\item $\RR\otimes_\RR \M{2^{\frac{n}{2}}}{\RR}$, if $n \equiv 0 \pmod 8$,
\item $\left( \RR \oplus \RR \right) \otimes_\RR \M{2^{\frac{n-1}{2}}}{\RR}$, if $n \equiv 1 \pmod 8$,
\item $\M{2}{\RR} \otimes_\RR \M{2^{\frac{n-2}{2}}}{\RR}$, if $n \equiv 2 \pmod 8$,
\item $\M{2}{\CC} \otimes_\RR \M{2^{\frac{n-3}{2}}}{\RR}$, if $n \equiv 3 \pmod 8$,
\item $\M{2}{\HH} \otimes_\RR \M{2^{\frac{n-4}{2}}}{\RR}$, if $n \equiv 4 \pmod 8$,
\item $\left( \M{2}{\HH} \oplus \M{2}{\HH} \right) \otimes_\RR \M{2^{\frac{n-5}{2}}}{\RR}$, if $n \equiv 5 \pmod 8$,
\item $\M{4}{\HH} \otimes_\RR \M{2^{\frac{n-6}{2}}}{\RR}$, if $n \equiv 6 \pmod 8$,
\item $\M{8}{\CC} \otimes_\RR \M{2^{\frac{n-7}{2}}}{\RR}$, if $n \equiv 7 \pmod 8$.
\end{enumerate}
\end{proposition}

\begin{proof}
See e.g.\ \cite[Prop.~4.4.1 + Table~4.4.1]{Kobayashi/Yoshino:2005}.
\end{proof}

Since $C(V,q)$ is an associative algebra, it becomes a Lie algebra by setting $[A,B]:=AB-BA$. 
With this in mind, Proposition~\ref{prop:cartan-bott} implies the following:

\begin{corollary} \label{cor:cartan-bott-max-cpt}
For $n\geq 2$, the maximal semisimple compact Lie subalgebra of the Clifford algebra $C(\RR^n,q)$ is isomorphic to the following Lie algebra:
\begin{enumerate}
\setcounter{enumi}{-1}
\item $\mathfrak{so}(2^{\frac{n}{2}})$, if $n \equiv 0 \pmod 8$,
\item $\mathfrak{so}(2^{\frac{n-1}{2}})\oplus\mathfrak{so}(2^{\frac{n-1}{2}})$, if $n \equiv 1 \pmod 8$,
\item $\mathfrak{so}(2^{\frac{n}{2}})$, if $n \equiv 2 \pmod 8$,
\item $\mathfrak{su}(2^{\frac{n-1}{2}})$, if $n \equiv 3 \pmod 8$,
\item $\mathfrak{sp}(2^{\frac{n-2}{2}})$, if $n \equiv 4 \pmod 8$,
\item $\mathfrak{sp}(2^{\frac{n-3}{2}})\oplus\mathfrak{sp}(2^{\frac{n-3}{2}})$, if $n \equiv 5 \pmod 8$,
\item $\mathfrak{sp}(2^{\frac{n-2}{2}})$, if $n \equiv 6 \pmod 8$,
\item $\mathfrak{su}(2^{\frac{n-1}{2}})$, if $n \equiv 7 \pmod 8$.
\end{enumerate}
\end{corollary}

\subsection{A lower bound on the dimension of a subalgebra}

\begin{definition}
For $n\geq 3$ let $\imp$ be the Lie subalgebra of $C(\RR^n,q)$ generated by $v_1v_2v_3$ and by $v_iv_{i+1}$, $1 \leq i < n$.
\end{definition}

\begin{lemma} \label{lowerbounddimension}
Let $n\geq 3$.
Then $\imp$ contains all products of the form $v_{j_1}v_{j_2}\cdots v_{j_k}$ for $2\leq k\leq n$ and $k \equiv 2, 3 \pmod 4$ with pairwise distinct $j_t \in \{ 1, \ldots, n \}$, with the possible exception of $v_1v_2\cdots v_n$, if $n \equiv 3 \pmod 4$. 
\end{lemma}

\begin{proof}
It is well-known that all products $v_{j_1}v_{j_2}$, $j_1 \neq j_2$, are contained in $\imp$: Indeed, $\Lambda^2 \RR^n \cong \mathfrak{so}(n)$ (cf., e.g., \cite[Prop.~6.1]{Lawson/Michelsohn:1989}) is generated as a Lie algebra by the $v_iv_{i+1}$, $1 \leq i < n$ (cf., e.g., \cite[Thm.~1.31]{Berman} and Theorem~\ref{Berman} of the main text).

Moreover, for pairwise distinct $j_t$, $1 \leq t \leq k+1$, one has 
\[[v_{j_1}v_{j_2},\;v_{j_2}v_{j_3}\cdots v_{j_{k+1}}] = 2v_{j_1}v_{j_3}\cdots v_{j_{k+1}}.\]
Since re-ordering of the factors simply yields scalar multiples, this shows inductively that, as long as $k+1 \leq n$, once an arbitrary factor of the form $v_{j_1}v_{j_2}\cdots v_{j_k}$ is contained in the Lie subalgebra, all factors of that form are contained in the Lie subalgebra. This statement is also true in the situation $k=n$, because in that case all factors of that form are scalar multiples of one another.   

We prove the claim of the lemma by induction over $k$. For $k=2$ and $k=3$, this is obvious.
Suppose the claim holds for $k\equiv 3\pmod 4$, so that the next value for $k$ to consider
is $k+3\equiv 2\pmod 4$. By induction hypothesis
$v_4v_5\cdots v_{k+3} \in \imp$ and
\[0 \neq [v_1v_2v_3,v_4v_5\cdots v_{k+3}] = 2 v_1v_2v_3v_4\cdots v_{k+3}. \]
If on the other hand the claim holds for $k\equiv 2\pmod 4$, then the next value for $k$ to consider
is $k+1\equiv 3\pmod 4$. If $k+2\leq n$, then by
induction hypothesis $v_3v_4\cdots v_{k+2}\in\imp$ and
\[0 \neq [v_1v_2v_3,v_3v_4\cdots v_{k+2}] = 2v_1v_2v_4\cdots v_{k+2}.\]
That is, the presence of all elements of the form $v_{j_1}v_{j_2}v_{j_3}$ with pairwise distinct $j_t \in \{ 1, \ldots, n \}$ inductively allows us to construct all elements of the form $v_{j_1}v_{j_2}\cdots v_{j_k}$ for $k \equiv 2, 3 \pmod 4$ with pairwise distinct $j_t \in \{ 1, \ldots, n \}$ for all $k \leq n$, with the possible exception of the situation $k=n \equiv 3 \pmod 4$, as the element $v_{k+2}$ does not exist in that case.
\end{proof}

\begin{remark}
It will turn out later, as a consequence of the proof of Theorem~\ref{thm:A} based on dimension arguments, that the above elements in fact generate $\imp$ as an $\RR$-vector space and that for $n \equiv 3 \pmod 4$ the element $v_1v_2\cdots v_n$ indeed is not contained in $\imp$, unless of course $n=3$.
\end{remark}

\begin{definition}
For $k\in\{0,1,2,3\}$, let
\[ \delta_k : \NN\to \NN : n \mapsto \sum_{\substack{i=0, \\ i \equiv k \pmod 4}}^n \binom{n}{i}. \]
\end{definition}


\begin{consequence} \label{con:lowerbounddim-ineffective}
Let $n\geq 3$.
Then
\[
\dim \imp \geq
\begin{cases}
\delta_2(n) + \delta_3(n) & \text{ if }  n \not\equiv 3 \pmod 4 , \\
\delta_2(n) + \delta_3(n)-1 & \text{ if }  n \equiv 3 \pmod 4 .
\end{cases}
\]
\end{consequence}

\subsection{Combinatorics of binomial coefficients}

We now turn the lower bound from Consequence~\ref{con:lowerbounddim-ineffective} into a numerically explicit bound by deriving a closed formula in $n$ for the functions $\delta_k$.

\begin{proposition} \label{binomial}
Let $n \in \NN$ and $k\in\{0,1,2,3\}$.
\begin{enumerate}
\setcounter{enumi}{-1}
\item If $n \equiv 0 \pmod 4$, then
\[ \delta_k(n)
= \begin{cases}
  2^{n-2}  & \text{for } k \in \{ 1, 3 \}, \\
  2^{n-2}+(-1)^{\frac{n}{4}+ \frac{k}{2}} 2^{\frac{n}{2}-1}
           & \text{for }k \in \{ 0, 2 \}.
\end{cases}\]
\item If $n \equiv 1 \pmod 4$, then
\[ \delta_k(n)
= \begin{cases} 
  2^{n-2}+(-1)^{\frac{n-1}{4}}2^{\frac{n-3}{2}} & \text{for } k \in \{ 0, 1 \}, \\
  2^{n-2}-(-1)^{\frac{n-1}{4}}2^{\frac{n-3}{2}} & \text{for } k \in \{ 2, 3 \}.
  \end{cases}
\]
\item If $n \equiv 2 \pmod 4$, then
\[ \delta_k(n)
= \begin{cases} 
  2^{n-2} & \text{for } k \in \{ 0, 2 \}, \\
  2^{n-2}+(-1)^{\frac{n-2}{4}+ \frac{k-1}{2}} 2^{\frac{n}{2}-1}
          & \text{for } k \in \{ 1, 3 \}.
  \end{cases}
\]
\item If $n \equiv 3 \pmod 4$, then
\[ \delta_k(n)
= \begin{cases}         
  2^{n-2}-(-1)^{\frac{n-3}{4}}2^{\frac{n-3}{2}} & \text{for } k \in \{ 0, 3 \}, \\
  2^{n-2}+(-1)^{\frac{n-3}{4}}2^{\frac{n-3}{2}} & \text{for } k \in \{ 1, 2 \}.
  \end{cases}
\]
\end{enumerate}
\end{proposition}

\begin{proof}
For $a, n \in \mathbb{N}$ the binomial theorem implies $$(1+i^a)^n = \sum_{k=0}^3 i^{ak}\delta_k(n),$$ where $i \in \mathbb{C}$ denotes the imaginary unit. Evaluation of this formula for $a \in \{ 0, 1, 2, 3 \}$ yields the following system of four identities:
\begin{align}
\delta_0(n) + \delta_1(n) + \delta_2(n) + \delta_3(n) &= 2^n, \label{1}\\
\delta_0(n) + i\delta_1(n) - \delta_2(n) -i \delta_3(n) &= (1+i)^n = 2^{\frac{n}{2}} \cdot e^{\frac{n2\pi i}{8}}, \label{2}\\
\delta_0(n) - \delta_1(n) + \delta_2(n) - \delta_3(n) &=  0, \label{3}\\
\delta_0(n) -i \delta_1(n) - \delta_2(n) +i \delta_3(n) &=  (1-i)^n = 2^{\frac{n}{2}} \cdot e^{-\frac{n2\pi i}{8}}. \label{4}
\end{align}
These four identities imply
\begin{align}
\delta_0(n) + \delta_2(n) &= 2^{n-1}   &\text{(\ref{1}) plus (\ref{3}) divided by $2$}, \label{5} \\ 
\delta_0(n) - \delta_2(n) &= 2^{\frac{n-2}{2}} (e^{\frac{n2\pi i}{8}}+e^{-\frac{n2\pi i}{8}}) &\text{(\ref{2}) plus (\ref{4}) divided by $2$}, \label{6} \\ 
\delta_1(n) + \delta_3(n) &=  2^{n-1}  &\text{(\ref{1}) minus (\ref{3}) divided by $2$}, \label{7} \\ 
\delta_1(n) - \delta_3(n) &= -2^{\frac{n-2}{2}}i (e^{\frac{n2\pi i}{8}}-e^{-\frac{n2\pi i}{8}}) &\text{(\ref{2}) minus (\ref{4}) divided by $2i$}.  \label{8}
\end{align}
One readily computes $\delta_0(n)$, $\delta_2(n)$ from (\ref{5}), (\ref{6}) and $\delta_1(n)$, $\delta_3(n)$ from (\ref{7}), (\ref{8}).
\end{proof}

Combining this with Consequence~\ref{con:lowerbounddim-ineffective} yields
the following:

\begin{consequence} \label{comparedimension}
Let $n \in \NN$ and $n\geq 2$.
\begin{enumerate}
\setcounter{enumi}{-1}
\item If $n \equiv 0 \pmod 8$, then
\begin{align*}
  \dim \imp
  \geq \delta_2(n) + \delta_3(n)
&= 2^{n-2} - 2^{\frac{n}{2}-1} + 2^{n-2}
= 2^{\frac{n-2}{2}}(2^{\frac{n}{2}}-1) \\
&= \dim_\RR(\mathfrak{so}(2^{\frac{n}{2}})).
\end{align*}
\item If $n \equiv 1 \pmod 8$, then
\begin{align*}
  \dim \imp
  \geq \delta_2(n) + \delta_3(n)
&= 2\left( 2^{n-2} - 2^{\frac{n-3}{2}} \right)
= 2^{\frac{n-1}{2}}(2^{\frac{n-1}{2}}-1) \\
&= \dim_\RR(\mathfrak{so}(2^{\frac{n-1}{2}}) \oplus \mathfrak{so}(2^{\frac{n-1}{2}})).
\end{align*}
\item If $n \equiv 2 \pmod 8$, then
\begin{align*}
  \dim \imp
  \geq \delta_2(n) + \delta_3(n)
&= 2^{n-2} + 2^{n-2} - 2^{\frac{n}{2}-1}
= 2^{\frac{n-2}{2}}(2^{\frac{n}{2}}-1) \\
&= \dim_\RR(\mathfrak{so}(2^{\frac{n}{2}})).
\end{align*}
\item If $n \equiv 3 \pmod 8$, then
\begin{align*}
  \dim \imp + 1
  \geq \delta_2(n) + \delta_3(n)
&= 2^{n-2} + 2^{\frac{n-3}{2}} + 2^{n-2} - 2^{\frac{n-3}{2}}
= 2^{n-1} \\
&= \dim_\RR(\mathfrak{su}(2^{\frac{n-1}{2}}))+1.
\end{align*}
\item If $n \equiv 4 \pmod 8$, then
\begin{align*}
  \dim \imp
  \geq \delta_2(n) + \delta_3(n)
&= 2^{n-2} + 2^{\frac{n}{2}-1} + 2^{n-2}
= 2^{\frac{n-2}{2}}(2^{\frac{n}{2}}+1) \\
&= \dim_\RR(\mathfrak{sp}(2^{\frac{n-2}{2}})).
\end{align*}
\item If $n \equiv 5 \pmod 8$, then
\begin{align*}
  \dim \imp
  \geq \delta_2(n) + \delta_3(n)
&= 2\left( 2^{n-2} + 2^{\frac{n-3}{2}} \right)
= 2^{\frac{n-1}{2}}(2^{\frac{n-1}{2}}+1) \\
&= \dim_\RR(\mathfrak{sp}(2^{\frac{n-3}{2}})\oplus \mathfrak{sp}(2^{\frac{n-3}{2}})).
\end{align*}
\item If $n \equiv 6 \pmod 8$, then
\begin{align*}
  \dim \imp
  \geq \delta_2(n) + \delta_3(n)
&= 2^{n-2} + 2^{n-2} + 2^{\frac{n}{2}-1}
= 2^{\frac{n-2}{2}}(2^{\frac{n}{2}}+1) \\
&= \dim_\RR(\mathfrak{sp}(2^{\frac{n-2}{2}})).
\end{align*}
\item If $n \equiv 7 \pmod 8$, then
\begin{align*}
  \dim \imp + 1
  \geq \delta_2(n) + \delta_3(n)
&= 2^{n-2} - 2^{\frac{n-3}{2}} + 2^{n-2} + 2^{\frac{n-3}{2}}
= 2^{n-1} \\
&= \dim_\RR(\mathfrak{su}(2^{\frac{n-1}{2}}))+1.
\end{align*}
\end{enumerate}
\end{consequence}

\subsection{Generalized spin representations of the split real $E_n$ series and the resulting quotients}

The example of a generalized spin representation of the maximal compact subalgebra of the split real Kac--Moody Lie algebra of type $E_{10}$ described in \cite{DamourKleinschmidtNicolai} and \cite{deBuylHenneauxPaulot} (see Example~\ref{theexample} in the main text) generalizes directly to the whole $E_n$ series as follows.

Let $n \in \NN$, let $\mathfrak{g}$ be the split real Kac--Moody Lie algebra of type $E_n$, let $\mathfrak{k}$ be its maximal compact subalgebra, and let $X_i$, $1 \leq i \leq n$, be the Berman generators of $\mathfrak{k}$ (cf.\ \cite[Thm.~1.31]{Berman} and Theorem~\ref{Berman} in the main text) enumerated in Bourbaki style as shown in  Figure~\ref{fig:En}, i.e., $X_1$, $X_3$, $X_4$, \ldots, $X_n$ belong to the $A_{n-1}$ subdiagram, generating $\mathfrak{so}(n)$, and $X_2$ to the additional node. As in Section~\ref{sec:cartan-bott} let $q$ be the standard positive definite quadratic form on $\RR^n$ and let $C = C(\RR^n,q)$ be the corresponding Clifford algebra, considered as a Lie algebra.

\begin{proposition}
Let $n\geq 3$. The assignment
\begin{itemize}
\item $X_{1} \mapsto \frac{1}{2}v_1v_2$,
\item $X_2 \mapsto \frac{1}{2}v_1v_2v_3$,
\item $X_j \mapsto \frac{1}{2}v_{j-1}v_j$ for $3 \leq j \leq n$
\end{itemize}
defines a Lie algebra homomorphism $\rho$ from $\mathfrak{k}$ to the Lie subalgebra $\imp$ of $C$ generated by $v_1v_2v_3$ and by $v_iv_{i+1}$, $1 \leq i < n$, called the {\bf standard generalized spin representation of $\mathfrak{k}$}.
\end{proposition} 

\begin{proof}
The proof is based on the criterion established in Remark~\ref{characterization} and is exactly the same as in the $E_{10}$ case discussed in Example~\ref{theexample}.  
\end{proof}

\begin{proof}[Proof of Theorem~\ref{thm:A}]
By Theorem~\ref{Mainexistencetheorem} and since
$E_n$ is simply laced and connected for $n\geq 4$, the image $\imp$ of $\rho$ is semisimple and compact. By Lemma~\ref{lowerbounddimension} and Consequence~\ref{comparedimension}, the dimension $\dim_\RR(\imp)$ is at least as large as the dimension of the maximal semisimple compact Lie subalgebra of $C$ as given in Corollary~\ref{cor:cartan-bott-max-cpt}. The claim follows. 
\end{proof}

\begin{proof}[Proof of Theorem~\ref{thm:B}]
Let $\mathfrak{g}$ be a semisimple split real Lie algebra of type $E_4=A_4$, $E_5=D_5$, $E_6$, $E_7$ or $E_8$ and $\mathfrak{g} = \mathfrak{k} \oplus \mathfrak{a} \oplus \mathfrak{n}$ its Iwasawa decomposition. Since $\dim_\mathbb{R}(\mathfrak{k}) = \dim_\mathbb{R}(\mathfrak{n})$, from the combinatorics of the respective root system we conclude that the maximal compact Lie subalgebra $\mathfrak{k}$ has dimension
\begin{align*}
10 &= \frac{4\cdot 5}{2}
    = \frac{2^{\frac{4}{2}} \cdot (2^{\frac{4}{2}}+1)}{2}
    = \dim_\RR(\mathfrak{sp}(2))
    = \dim_\RR(\mathfrak{so}(5))
	&\text{ if } n=4, \\
20 &= 2 \cdot 10
    = \dim_\RR(\mathfrak{sp}(2)\oplus\mathfrak{sp}(2))
    = \dim_\RR(\mathfrak{so}(5)\oplus\mathfrak{so}(5))
	&\text{ if } n=5, \\
36 &= 4 \cdot 9
    = 2^{\frac{6-2}{2}}\cdot(2^{\frac{6}{2}}+1)
    = \dim_\RR(\mathfrak{sp}(4)) 
	&\text{ if } n=6, \\
63 &= 2^6-1
    = \dim_\RR(\mathfrak{su}(8)) 
	&\text{ if } n=7, \\
120 &= \frac{16 \cdot 15}{2}
    = \frac{2^{\frac{8}{2}} \cdot (2^{\frac{8}{2}}-1)}{2}
    = \dim_\RR(\mathfrak{so}(16)) 
	&\text{ if } n=8.
\end{align*}
For $n\geq 4$ we may now apply Theorem~\ref{thm:A} and deduce that the
standard generalized spin representation $\rho$ has to be injective
in these cases.

This leaves the case $E_3=A_2\oplus A_1$. Since this diagram is not irreducible,
Theorem~\ref{Mainexistencetheorem}
only implies that $\im(\rho)=\imp$ is compact but not that it is
semisimple (and, indeed, it is not). However, $n=3$ is also 
an exceptional case for Lemma~\ref{lowerbounddimension}: In this case $\dim_\RR(\imp)=4$, as $v_1v_2$, $v_1v_3$, $v_2v_3$, $v_1v_2v_3$ form an $\RR$-basis of $\imp$. 
On the other hand, the Clifford algebra $C$ is isomorphic to $\M{2}{\CC}$,
hence $\kk\cong \mathfrak{u}(2)$, and this has dimension $4$.
Thus $\rho$ is also injective when $n=3$. The claim follows.
\end{proof}


\noindent 
JLU Giessen, Mathematisches Institut, Arndtstrasse 2, 35392 Giessen, Germany \\
{\tt ralf.koehl@math.uni-giessen.de } \\
{\tt max.horn@math.uni-giessen.de }

\noindent
Mathematics and Statistics, Fylde College, Lancaster University, Lancaster LA1 4YF, United Kingdom \\
{\tt p.d.levy@lancaster.ac.uk}

\end{document}